\documentclass[11pt]{amsart}
\usepackage[official]{eurosym}
\usepackage{babel}
\usepackage{bbm}
\usepackage[utf8]{inputenc}
\usepackage[T1]{fontenc}
\usepackage{pdfpages}
\usepackage{tensor}
\usepackage{enumerate}
\usepackage[stable]{footmisc}
\usepackage{mathtools}
\usepackage[mathscr]{eucal}
\usepackage[a4paper, twoside=true, hmargin={2cm,3cm}, vmargin={1.5 cm,1.5cm}, includehead]{geometry}
\usepackage{pgfgantt}
\usepackage{graphicx}
\usepackage{xcolor}
\ganttset{group/.append style={orange},
milestone/.append style={red},
progress label node anchor/.append style={text=red}}
\theoremstyle{plain}
\newtheorem{theorem}{Theorem}[section]
\newtheorem{lemma}[theorem]{Lemma}
\newtheorem{proposition}[theorem]{Proposition}

\newtheorem{conjecture}{Conjecture}
\newtheorem*{conjecture*}{Conjecture}

\newtheorem{definition}[theorem]{Definition}

\newtheorem{remark}{Remark}
\newtheorem*{Graal1*}{Special Case of Problem~\ref{Graal1}}
\newtheorem*{Graal2*}{Special Case of Problem~\ref{Graal2}}

\newtheorem*{theorem*}{Theorem}

\newtheorem{innercustomthm}{Theorem}
\newenvironment{customthm}[1]
  {\renewcommand\theinnercustomthm{#1}\innercustomthm}
  {\endinnercustomthm}

\theoremstyle{remark}

\newtheorem*{remark*}{{\bf Remark}}
\newtheorem*{remarks*}{{\bf Remarks}}
\newtheorem*{comment*}{{\bf Comment}}

\usepackage{amsmath}
\newcommand{\wt}{\widetilde}

\newcommand{\nnorm}[1]{\lvert\!|\!| #1|\!|\!\rvert}
\newcommand{\bignorm}[1]{\big\lVert #1 \big\rVert}
\newcommand{\norm}[1]{\lVert #1 \rVert}
\newcommand{\N}{\mathbb{N}}
\newcommand{\Q}{\mathbb{Q}}
\newcommand{\R}{\mathbb{R}}
\newcommand{\T}{\mathbb{T}}
\newcommand{\Z}{\mathbb{Z}}

\newcommand{\B}{\mathcal{B}}
\newcommand{\C}{\mathbb{C}}

\newcommand{\G}{\Gamma}

\newcommand{\X}{\mathcal{X}}

\newcommand{\m}{\mu}

\newcommand{\bol}{{\bf 1}}

\newcommand{\e}{\varepsilon}

\def \colon{{:}\;}
\DeclareMathOperator*{\E}{\mathbb{E}}
\DeclarePairedDelimiter\floor{\lfloor}{\rfloor}

\usepackage{fancyhdr}
\usepackage{theoremref}
\usepackage{hyperref}
\hypersetup{
    colorlinks=true,
    linkcolor=red,
    filecolor=magenta,      
    urlcolor=cyan,
}

\title{Pointwise convergence in nilmanifolds along smooth functions of polynomial growth}

\thanks{The author was supported by the Research Grant - ELIDEK HFRI-FM17-1684 and ELIDEK-Fellowship number 5367 (3rd Call for HFRI Ph.D. Fellowships) during the preparation of this article.}

\author{Konstantinos Tsinas}

\address[Konstantinos Tsinas]{University of Crete, Department of Mathematics and applied mathematics, Voutes University Campus, Heraklion 71003, Greece} \email{kon.tsinas@gmail.com}



\begin{document}

\maketitle

\begin{abstract}
    We study the equidistribution of orbits of the form $b_1^{a_1(n)}\cdots b_k^{a_k(n)}\G$ in a nilmanifold $X$, where the sequences $a_i(n)$ arise from smooth functions of polynomial growth belonging to a Hardy field. We show that under certain assumptions on the growth rates of the functions $a_1,...,a_k$, these orbits are equidistributed on some subnilmanifold of the space $X$. As an application of these results and in combination with the Host-Kra structure theorem for measure preserving systems, as well as some recent seminorm estimates of the author for ergodic averages concerning Hardy field functions, we deduce a norm convergence result for multiple ergodic averages. 
    Our method mainly relies on an equidistribution result of Green-Tao on finite segments of polynomial orbits on a nilmanifold.
\end{abstract}

\section{Introduction and main results}\label{introduction}

\subsection{History and main goals} In recent years, there has been an active interest in determining the limiting behavior of the multiple ergodic averages \begin{equation}\label{multiple}
    \frac{1}{N}\sum_{n=1}^{N} f_1(T^{a_1(n)}x)\cdots f_k(T^{a_k(n)}x)
\end{equation}for various sequences $a_1(n),...,a_k(n)$ of integers, where $T$ is an invertible measure preserving transformation acting on a probability space $(X,\mathcal{X},\m)$ and $f_1,..,f_k$ are functions in $L^{\infty}(\m)$. Through the breakthrough work of Furstenberg \cite{Fu1}, which delivered a new proof of Szemer\'{e}di's theorem using tools from ergodic theory, it has been apparent that the analysis of the averages in \eqref{multiple} has noteworthy applications to number theory and combinatorics. In particular, we now have substantial generalizations of Szemer\'{e}di's theorem, some of which have not been demonstrated with approaches other than the use of ergodic theory.

An integral tool in verifying convergence of the averages in \eqref{multiple} is the structure theorem of Host-Kra \cite{Host-Kra1}, which in multiple cases reduces the above problem to studying rotations on particular spaces with algebraic structure, which are called nilmanifolds (see \cite{Host-Kra structures} for a full presentation of the theory). A nilmanifold is a homogeneous space $X=G/\G$, where $G$ is a nilpotent Lie group and $\G$ is a discrete cocompact subgroup.
The study of nilmanifolds is essential due to its ties to ergodic theory mentioned above, as well as the numerous applications to combinatorics and number theory.

In this article, our central problem is the study of the distribution of orbits in a nilmanifold along sequences that arise from smooth functions with polynomial growth. We suppose that our functions are elements of a Hardy field (for the definition of a Hardy field, we direct the reader to Section \ref{background}). 
The benefit of working within a Hardy field is that certain "regularity" properties of the derivatives of a function, which are vital in several parts of our proofs, can be extrapolated from a simple growth condition on the initial function. For instance, a condition like \eqref{P_0} below imposes multiple pleasant properties on the derivatives of a function in $\mathcal{H}$.

The field of logarithmico-exponential functions is the prototypical example of a Hardy field. It is defined as the collection of functions formed by a finite combination of the operations $+,-,\cdot,\div,\  \exp$, $\log$ and composition of functions acting on a real variable $t$ (which takes values on some half-line $[x,+\infty)$) and real constants. The fact that it is a Hardy field was established in \cite{Hardy 1}. Our results are most interesting for the Hardy field $\mathcal{LE}$ and one can keep this particular case in mind throughout the article. In addition, we refer the reader to Appendix \ref{nilpotent crap} for the definition and properties of nilmanifolds, which appear in the subsequent discussion and the main theorems.

Due to its connections to ergodic theory and combinatorics, the investigation of equidistribution properties along Hardy sequences has been carried out several times throughout the literature.
First of all, we recall a fundamental result concerning the equidistribution of Hardy sequences, which corresponds to the basic case when the underlying nilmanifold is a finite-dimensional torus. In particular, we restate here Theorem 1.3 from \cite{Boshernitzan1}:

\begin{customthm}{A}[Boshernitzan]\label{Bosh}
Let the function $a \in \mathcal{H}$ have polynomial growth. Then, the sequence $a(n)$ is equidistributed $\mod 1$ if and only if\begin{equation}\label{P_0}
   \tag{P} \lim\limits_{t\to+\infty} \frac{|a(t) -p(t)   |}{\log t}=+\infty    \ \text{ for any polynomial }\ p(t)\in \Q [t].
\end{equation}
\end{customthm}

Applying Weyl's equidistribution theorem and the previous result, we can effortlessly show that if the functions $a_1,...,a_k$ have polynomial growth and each non-trivial linear combination of them stays logarithmically away from real multiples of integer polynomials, then the sequence $(c_1a_1(n),...,c_ka_k(n))$ is equidistributed on $\T^k$ for all non-zero real numbers $c_1,...,c_k$. Practically, Theorem \ref{Bosh} can be used to examine orbits on $\T^k$ along the sequences $a_1(n),\dots, a_k(n)$ on $\T^k$, answering our problem in the case when the nilmanifold $X$ is any finite-dimensional torus (the abelian case).
Another corollary of Theorem \ref{Bosh} is that if $a(t)\in \mathcal{H}$ stays logarithmically away from real multiples of integer polynomials, then the sequence $\floor{a(n)}a$ is equidistributed on $\T$ for all irrational $a\in (0,1)$. This phenomenon (namely, that equidistribution properties of $a(n)$ yield information for the equidistribution properties of $\floor{a(n)}$)
will be present throughout the article, so the reader can view statements involving $a(n)$ in place of $\floor{a(n)}$ as being morally the same.

Suppose now that we are given a nilmanifold $X=G/\G$ (for the definitions of all terms below, see Appendix \ref{nilpotent crap})
and assume that the group $G$ is connected and simply connected. We are interested in the behavior of the sequence \begin{equation}\label{g shit}
    v(n)=(b_1^{\floor{a_1(n)}}\G,\dots,b_k^{\floor{a_k(n)}}\G),
\end{equation}
where $b_1,\dots ,b_k$ are elements of the group $G$ and $a_1,\dots,a_k$ are Hardy field functions. Notice that this is a sequence on the product nilmanifold $X^k$.
The most fundamental equidistribution result is due to Leibman, who showed that if the functions $a_1,\dots a_k$ are integer polynomials, then we have equidistribution on a "subspace" of $X$ (called a subnilmanifold), as long as we restrict the values of $n$ to appropriate arithmetic progressions.

More specifically, we present the following theorem \cite[Theorem B]{Leibmanpointwise}, an application of which (on the nilmanifold $X^k$) implies the claim in the previous paragraph. 

\begin{customthm}{B}[Leibman]\label{Leibmantheorem}
Let $X=G/\G$ be a nilmanifold and $x\in X$. Consider the sequence  \begin{equation}
    g(n)=b_1^{p_1(n)}\dots b_k^{p_k(n)}
\end{equation}
in $G$, where $b_1,\dots,b_k\in G$ and $p_1,\dots, p_k$ are polynomials with integer coefficients. Then, there exists $Q\in \N$, a closed, connected and rational subgroup $H$ of $G$ and points $x_0,\dots x_{Q-1}\in X$, such that for every $r\in \{0,\dots ,Q-1\}$ the sequence $g(Qn+r)x$ is equidistributed on the subnilmanifold $Hx_r$.
\end{customthm}

A noteworthy corollary of the previous theorem is that if $F: X\to\C$ is a continuous function, then the averages \begin{equation*}
    \frac{1}{N}\sum_{n=1}^{N} F(g(n)x)
\end{equation*}converge pointwise for all $x\in X$. This can be used in conjunction with the Host-Kra structure theory (see Theorem \ref{structure} in Section \ref{background}) to infer that the averages in \eqref{multiple} converge in norm, when the sequences $a_1(n),...,a_k(n)$ are integer polynomial sequences. In addition, we deduce (as a corollary of \cite[Theorem C]{Leibmanpointwise} in the same paper) that if $G$ is connected, the equidistribution of the sequence $g(n)\G$ is controlled by the projection of $g(n)\G$ on the "abelianization" $G/[G, G]\G$ of $G/\G$, which is a finite-dimensional torus called the horizontal torus of $X$. 

A major improvement\footnote{While their theorem was established under the stronger hypothesis that the underlying Lie group $G$ is connected and simply connected, one can typically reduce to this case in many applications.} of the above theorem was established by Green and Tao in \cite{Green-Tao}, who characterized the behavior of polynomial orbits on nilmanifolds in quantitative language. This theorem has notable applications in number theory and will be undoubtedly vital in this paper. Like Leibman's theorem in \cite{Leibmanpointwise} that we mentioned above briefly, this theorem highlights the relation of the equidistribution properties of a polynomial sequence (see Definition \ref{filtration}) on a nilmanifold with its projection to the horizontal torus.
Since there are many technical terms that are required in order to state this theorem, we have presented its statement in Appendix \ref{nilpotent crap} along with a sample corollary when the nilmanifold is a torus, as well as all of the required background on the quantitative equidistribution theory on nilmanifolds.

Now, let us consider the more general case when the sequences $a_1(n),\dots a_k(n)$ appearing in \eqref{g shit} are not just integer polynomials, but functions that belong to a Hardy field $\mathcal{H}$.
In the case $k=1$, Frantzikinakis established \cite{Fraeq} that if the function $a(t)$ satisfies \begin{equation*}
   \lim\limits_{t\to+\infty} \frac{|a(t) -cp(t)   |}{\log t}=+\infty    \ \text{ for any polynomial }\ p(t)\in \Z [t],
\end{equation*}then the sequence $b^{\floor{a_1(n)}}x$ is equidistributed on the orbit $Y=\overline{\{b^nx\colon n\in \N\}}$ of $b$ for any $b\in G$ and $x\in X$. In the case of general $k$, he also established the next theorem in the same paper:

\begin{customthm}{C}[Frantzikinakis]\cite[Theorem 1.3]{Fraeq}\label{Fratheorem}
    Let $a_1,\dots, a_k$ be functions of polynomial growth that belong to a Hardy field $\mathcal{H}$, such that they have pairwise distinct growth rates and satisfy \begin{equation}\label{Fragrowthcondition}
        t^{k_i}\log t\prec a_i(t)\prec t^{k_i+1}
    \end{equation}for some $k_i\in \N$. Then, for any nilmanifold $X=G/\G$ and $b_1,\dots, b_k\in G$, the sequence \begin{equation}\label{productsequence}
        \Big(b_1^{\floor{a_1(n)}}x_1,\dots, b_k^{\floor{a_k(n)}}x_k   \Big)_{n\in \N}
    \end{equation}is equidistributed on $\overline{(b_1^{n}x_1)}_{n\in \N}\times \dots \times \overline{(b_k^nx_k)}_{n\in \N}$ for all $x_1,\dots, x_k\in X$.
\end{customthm}

In the same paper, Frantzikinakis conjectured that if the linear combinations of the functions $a_1,\dots,a_k$ stay logarithmically away from real multiples of integer polynomials, then the sequence in \eqref{productsequence} is equidistributed on  $\overline{(b_1^{n}x_1)_{n\in \N}}\times \dots \times \overline{(b_k^nx_k)_{n\in \N}}$. More specifically, we have the following:

\begin{conjecture}\cite{Fraeq}
      Let $a_1,\dots, a_k$ be functions in a Hardy field $\mathcal{H}$ with polynomial growth and such that every non-trivial linear combination $a(t)$ of them satisfies \begin{equation*}
            \lim\limits_{t\to+\infty} \frac{|a(t) -p(t)   |}{\log t}=+\infty    \ \text{ for any polynomial }\ p(t)\in \Z[t].  
      \end{equation*}Then, for any nilmanifold $X=G/\G,\ b_i\in G$ and $x_i\in X$, the sequence\begin{equation*}
           \Big(b_1^{\floor{a_1(n)}}x_1,\dots, b_k^{\floor{a_k(n)}}x_k   \Big)_{n\in \N}
      \end{equation*}is equidistributed on $\overline{(b_1^{n}x_1)}_{n\in \N}\times \dots \times \overline{(b_k^nx_k)}_{n\in \N}$.
\end{conjecture}

Recently, Richter established the following equidistribution theorem. We present here a special case of that result, where we assume that the underlying Lie group $G$ is connected and simply connected so that the elements $b^s$ are defined for any $b\in G$ and $s\in \R$ (see also the first paragraph of Subsection \ref{bs-Ratner subsection } for a more thorough explanation). We also define \begin{equation*}
    \nabla-\text{span}\{a_1,\dots,a_k\}=\{c_1a_1^{(n_1)}(t)+\dots +c_ka_k^{(n_k)}(t)\colon  \ c_i\in \R, \ n_i\in \N\cup \{0\}\}.
\end{equation*} 
\begin{customthm}{D}[Richter]\cite[Theorem B]{Richter}\label{Richtertheorem}
    Let $X=G/\G$ be a nilmanifold with $G$ connected and simply connected and let $a_1,\dots, a_k$ be functions in a Hardy field $\mathcal{H}$, such that for any function $a\in \nabla-\text{span}\{a_1,\dots,a_k\}$, we have that $$|a(t)-p(t)|\ll 1 \ \text{ or } |a(t)-p(t)|\succ \log t,$$for any polynomial $p(t)\in \R[t]$.
    Consider any commuting elements $b_1,\dots,b_k\in G$ and define the sequence \begin{equation*}
        v(n)=b_1^{a_1(n)}\dots b_k^{a_k(n)}.
    \end{equation*}Then, there exists a closed, connected, and rational subgroup $H$ of $G$ and points $x_0,\dots x_{Q-1}$ in $X$, such that the sequence $v(Qn+r)\G$ is equidistributed on the subnilmanifold $Hx_r$ of $X$ for all $r\in \{0,\dots, Q-1\}$.
\end{customthm}The hypothesis that $b_1,\dots, b_k$ are commuting is harmless in problems regarding the convergence of ergodic averages or in applications to combinatorics. Furthermore, while in this setting we have the sequences $a_i(n)$ instead of $\floor{a_i(n)}$ in the exponents, the statement above actually implies an equidistribution theorem for the sequences $\floor{a_i(n)}$. We remark that the results in \cite{Richter} are generalized to equidistribution results with respect to (weaker) averaging schemes other than Ces\'{a}ro averages. Under those averaging schemes, the assumptions on the functions $a_1,\dots, a_k$ can be weakened significantly (on the other hand, our results deal only with Ces\'{a}ro averages). In the follow-up paper \cite{BerMorRic}, Bergelson, Moreira and Richter employed the above equidistribution results to obtain convergence results for multiple ergodic averages and combinatorial applications for Hardy field sequences.

\subsection{Main results}

In order to state our results, we will assume that we have a fixed Hardy field $\mathcal{H}$, and the only extra hypothesis we require is that it includes the polynomial functions (this is a very mild restriction). Removing this restriction may be possible, though this would certainly complicate our arguments or the notation in the proofs. Unless noted otherwise, our theorems below apply to any such Hardy field. An exception is made only for Theorem \ref{simple} (we shall reiterate these assumptions in the main theorems).

For a given set of functions $a_1,...,a_k$ in our Hardy field $\mathcal{H}$, we use the notation \begin{equation}\label{mathcal{L}}
    \mathcal{L}(a_1,...,a_k)=\{c_1a_1+\dots +c_ka_k\colon \  (c_1,...,c_k)\in \R^{k}\setminus\{0\}\}
\end{equation}
to refer to the collection of functions in $\mathcal{H}$ that are non-trivial linear combinations of the functions $a_1(t),...,a_k(t)$. The nilmanifolds $\overline{(b^{\R}x)}$ and $\overline{(b^{\N}x)}$ are defined in Subsection \ref{bs-Ratner subsection }.

\begin{theorem}\label{equidistribution}Let $\mathcal{H}$ be a Hardy field containing the polynomial functions.
Let $a_1,...,a_k$ be functions in $\mathcal{H}$ that have polynomial growth. Assume that there exists\footnote{The value of $\e$ depends only on the initial collection $\{a_1,...,a_k\}$.} an $\e>0$, such that every function $a\in \mathcal{L}(a_1,...,a_k)$ satisfies\footnote{Equivalently, we could require that $p(t)\in \Z[t]$, because this is a condition on all the linear combinations of the functions $a_1,...,a_k$. } 
\begin{equation}\label{P0}
    \lim\limits_{t\to+\infty} \frac{|a(t) -p(t)   |}{t^{\e}}=+\infty    \ \text{ for any polynomial }\ p(t)\in \Q[t].
\end{equation} Then, we have the following:\\
 (i) For any collection of nilmanifolds $X_i=G_i/\Gamma_i$, elements $b_i\in G_i$ and $x_i\in X_i$, the sequence 
    \begin{equation*}
        \big( b_1^{\floor{a_1(n)}}x_1,...,b_k^{\floor{a_k(n)}}x_k    \big)
    \end{equation*}is equidistributed on the nilmanifold $\overline{(b_1^{\N} x_1)}\times\dots \times \overline{(b_k^{\N} x_k)}$.\\
    (ii) For any collection of nilmanifolds $X_i=G_i/\Gamma_i$ such that the groups $G_i$ are connected, simply connected, elements $b_i\in G_i$ and $x_i\in X_i$, the sequence 
    \begin{equation*}
        \big( b_1^{a_1(n)}x_1,...,b_k^{a_k(n)}x_k   \big)
    \end{equation*}is equidistributed on the nilmanifold $\overline{(b_1^{\R} x_1)} \times\dots \times \overline{(b_k^{\R} x_k)}$.
\end{theorem}
\begin{remark}
a) The connectedness assumptions imposed on the second part of the previous theorem ensure that all elements of the form $b^s$ where $b\in G$ and $s\in \R$ are well defined (see also Appendix \ref{nilpotent crap} for the definition of the element $b^s$ for non-integer $s$). \\
b) In regards to part ii) of the previous theorem, we establish the more general statement that if $b_1,\dots,b_k$ commute, the sequence $b_1^{a_1(n)}\cdots b_k^{a_k(n)}\G$ is equidistributed on the nilmanifold $\overline{b_1^{\R}\cdots b_k^{\R}\G}$. The fact that this is indeed a more general statement can be seen by passing to the product nilmanifold $X_1\times \dots \times X_k$. A similar assertion holds for Theorem \ref{pointwise} below and we provide more details on this deduction after Proposition \ref{impprop}. 
\end{remark}

Observe that, in contrast to Theorem \ref{Bosh}, we have the term $t^{\e}$ in the denominator, which is just out of reach of the conjectured optimal term $\log t$. As an example, using Theorem \ref{equidistribution}, we can prove that for any elements $b_1,b_2\in G$, the sequence $(b_1^{n\log n}\G,b_2^{n^{3/2}}\G)$ is equidistributed on the nilmanifold $(\overline{b_1^{\R}\G},\overline{b_2^{\R}\G})$, assuming that $G$ satisfies the appropriate connectedness assumptions since we want these elements to be well defined.

If we have functions that are not linearly independent, then the above theorem fails, as can be seen by noting that the sequence $(n^{3/2},n^{1/2},n^{3/2}+n^{1/2})$ is not equidistributed on $\T^{3}$. However, we can relax the linear independence condition in Theorem \ref{equidistribution} and still obtain a convergence result:
 \begin{theorem}\label{pointwise}Let $\mathcal{H}$ be a Hardy field containing the polynomial functions.
 Let $a_1,...,a_k$ be functions in $\mathcal{H}$ that have polynomial growth. Assume that there exists $\e>0$, such that every function $a\in \mathcal{L}(a_1,...,a_k)$ satisfies either
\begin{equation}\label{P1}
    \lim\limits_{t\to+\infty} \frac{|a(t) -p(t)   |}{t^{\e}}=+\infty    \ \text{ for any polynomial }\ p(t)\in \Q[t],
\end{equation} or \begin{equation}\label{P2}
  \text{the limit }\  \lim\limits_{t\to+\infty} a(t)  \ \text{is a real number}.
    \end{equation}Then, we have the following:\\
    (i) For any collection of nilmanifolds $X_i=G_i/\Gamma_i$, elements $b_i\in G_i$, $x_i\in X_i$ and continuous functions $f_1,...,f_k$ with complex values, the averages \begin{equation*}
        \frac{1}{N}\sum_{i=1}^{N} f_1(b_1^{\floor{a_1(n)}}x_1)\cdot ... \cdot f_k(b_k^{\floor{a_k(n)}}x_k)
    \end{equation*}converge.\\
(ii) For any collection of nilmanifolds $X_i=G_i/\Gamma_i$ such that the groups $G_i$ are connected, simply connected, elements $b_i\in G_i$, $x_i\in X_i$ and continuous functions $f_1,...,f_k$ with complex values, the averages \begin{equation*}
        \frac{1}{N}\sum_{i=1}^{N} f_1(b_1^{a_1(n)}x_1)\cdot ... \cdot f_k(b_k^{a_k(n)}x_k)
    \end{equation*}converge.
\end{theorem}
The main distinction between Theorems \ref{equidistribution} and \ref{pointwise} is that in the second case, we allow for linear dependencies between the functions $a_1(t),..., a_k(t)$ (for example, we may have the functions $(t\log t, t^{3/2},t^{3/2}+t\log t)$). We will use this theorem to deduce a convergence result for multiple ergodic averages (Theorem \ref{simple} below).

Theorems \ref{equidistribution} and \ref{pointwise} extend the equidistribution result of Theorem \ref{Fratheorem} from \cite{Fraeq}, where the functions $a_1,...,a_k$ were assumed to have different growth rates and satisfy the growth condition in \eqref{Fragrowthcondition}. On the other hand, our results are complementary to the results in \cite{Richter}, in the sense that both Theorem \ref{equidistribution} and Theorem \ref{Richtertheorem} each cover collections of functions that are not implied by the other one. The main difference between our results and the results in the previous literature (in the case of general $k$) is that prior results did not cover functions in the range $t^{\ell}\prec a(t)\ll  t^{\ell}\log t$, where $\ell$ is a positive integer. Our method circumvents this restriction and can handle all families of functions of the form $\sum_{i=1}^{k} c_it^{a_i}(\log t)^{b_i}$, where $a_i>0$ and $b_i,c_i\in \R$ (assuming, of course, that the linear combinations of the involved functions satisfy either \eqref{P1} or \eqref{P2}). However, our method has a drawback. As we stated, there are cases covered in the results of \cite{Richter} that do not follow from the arguments present in this paper. These examples concern functions that grow slower than fractional powers $t^{\delta}$, such as the function $\log^c t$ for $c>0$
or the function $\exp(\sqrt{\log t})$. An example that is not covered by Theorem \ref{pointwise} is the pair of functions $(\log ^2 t, t^{3/2})$. However, this last pair of functions can be covered by the results in \cite{Richter}. We shall discuss the techniques and limitations of our proof in depth below (Subsection 1.3).

Combining Theorem \ref{pointwise} and the results in \cite{tsinas} on characteristic factors, we get a mean convergence result for multiple ergodic averages. Since the seminorm estimates for such averages were established in \cite{tsinas} under particular assumptions on our Hardy field $\mathcal{H}$, these have to be incorporated into our statement. We will not need to use these assumptions anywhere else in this article, however.

\begin{theorem}\label{simple}Let $\mathcal{H}$ be a Hardy field that contains the field $\mathcal{LE}$ of logarithmico-exponential functions and is closed under composition and compositional inversion of functions (when defined). Furthermore, assume that the functions $a_1,...,a_k\in \mathcal{H}$ are as in Theorem \ref{pointwise}. Then, for any measure preserving system $(X,\m,T)$ and any functions $f_1,...,f_k\in L^{\infty}(\m)$, the averages \begin{equation}\label{ergodicaverages}
    \frac{1}{N}\sum_{n=1}^{N} T^{\floor{a_1(n)}}f_1\cdot ... \cdot T^{\floor{a_k(n)}}f_k
\end{equation}converge in $L^2(\m)$.
\end{theorem}

 An example of a Hardy field that satisfies the above property is the Hardy field of Pfaffian functions (for the definition, see, for instance, \cite[Section 2]{tsinas}).

It follows from the results in \cite{tsinas} that, if the functions $a_1,...,a_k$ are as in Theorem \ref{equidistribution} (actually, the $t^{\e}$ term can be replaced with the optimal term $\log t$), then for any ergodic measure preserving system $(X,\m,T)$ and bounded functions $f_1,...,f_k$, the averages $$ \frac{1}{N}\sum_{i=1}^{N} f_1(T^{\floor{a_1(n)}}x)\cdot...\cdot f_k(T^{\floor{a_k(n)}}x)$$ converge in the $L^2$-sense to the product of the integrals $\int f_1 d\m\cdot ...\cdot  \int f_k d\m$. The methods used in that article cannot work when there are linear dependencies between the functions $a_1,...,a_k$ (since they rely on the joint ergodicity results from \cite{Frajoint}). Therefore, in order to prove Theorem \ref{simple}, we have to show that the Host-Kra factors are characteristic for these averages, reduce the problem to nilmanifolds using the Host-Kra structure theorem (see Theorem \ref{structure} in Section \ref{background}) and then tackle the problem of mean convergence in nilmanifolds. The first part of the above argument follows from the results in \cite{tsinas} (see Proposition \ref{factors}), while Theorem \ref{pointwise} gives the stronger result of pointwise convergence when the system $(X,\m, T)$ is a nilsystem. We comment here that the optimal restrictions on the functions $a_1,...,a_k$ in Theorem \ref{simple} are expected to be that the functions are good for convergence when the system $(X,\m, T)$ is a rotation on some torus $\T^d$. A refuted conjecture of Frantzikinakis appears in \cite[Problem 22]{Fraopen}, although the statement needs to be changed to the following (personal communication):
\begin{conjecture}
Let $a_1,...,a_k$ be functions in $\mathcal{LE}$ (or any other Hardy field) such that for all real numbers $t_1,...,t_k\in [0,1)$, the averages \begin{equation}\label{exponentialaverages}
    \frac{1}{N}\sum_{n=1}^{N} e(t_1\floor{a_1(n)}+\dots +t_k\floor{a_k(n)})
\end{equation}converge. Then, for any measure preserving system $(X,\m,T)$ and functions $f_1,...,f_k\in L^{\infty}(\m)$, the averages \begin{equation}
    \frac{1}{N}\sum_{n=1}^{N} T^{\floor{a_1(n)}}f_1 \cdot... \cdot  T^{\floor{a_k(n)}}f_k
\end{equation}converge in $L^2(\m)$ and, if $(X,\m,T)$ is a nilsystem and the functions $f_1,...,f_k$ are continuous, then those averages converge pointwise everywhere.
\end{conjecture}

\begin{remark}It can be shown that the above condition on the exponentials of the involved sequences is not sufficient if we replace the Hardy sequences with other, more general, sequences. Indeed, \cite[Theorem B]{FraLeWie} provides an example of a sequence $a(n)$, such that for $a_1(n)=a(n)$ and $a_2(n)=2a(n)$, the averages \begin{equation*}
    \frac{1}{N}\sum_{n=1}^{N} e(t_1a_1(n)+t_2a_2(n))
\end{equation*}converge for any $t_1,t_2\in [0,1)$, but the ergodic averages \begin{equation*}
    \frac{1}{N}\sum_{n=1}^{N} T^{a_1(n)}f_1\cdot T^{a_2(n)}f_2
\end{equation*}do not converge in mean.
\end{remark}

\subsection{Short overview of the proof and additional remarks}
 The main idea of the proof is that functions in $\mathcal{H}$ of polynomial growth can be approximated sufficiently well by polynomials in short intervals.
The equidistribution properties of polynomial sequences in nilmanifolds, even on small intervals, are well understood from \cite{Green-Tao}.
We use these quantitative results of Green-Tao (see Theorem \ref{quantitative} in Appendix \ref{nilpotent crap}) to show that the averages over small intervals are "close" to the integral of a continuous function in our nilmanifold.
This approach was used in \cite{Fraeq} to show that (following the notation of Theorem \ref{equidistribution}) the sequence $b^{a(n)}$ is equidistributed for all $b\in G$ and any function $a(n)$ satisfying \eqref{P_0}. In the case that we need to cover here, there are more technical difficulties in the proof, since we have to find polynomial expansions for several functions simultaneously, which also tend to be of increased complexity (for example, choosing the length of the short intervals is fairly straightforward in the case of one function, but not when someone deals with several functions in $\mathcal{H}$). This idea of using a common polynomial expansion was also used recently by the author in order to establish the corresponding problem of finding characteristic factors for ergodic averages involving Hardy field iterates. This approach is well suited to handle functions in the range $t^k\prec a(t)\prec t^k\log t$, which were previously not known in the literature. Some additional care needs to be taken in order to separate polynomial functions and functions that we call "strongly non-polynomial" (see Definition \ref{growthdefinitions}). This is an elementary argument and is carried out in Lemma \ref{growthdecomposition} in Appendix \ref{hardyappendix}. A similar "decomposition" idea is present in \cite[Lemma A.3]{Richter} (also used in \cite{tsinas}), but we cannot use the exact same decomposition here, because some information on the linear combinations of our functions would be lost.

Our argument differs quite a bit from the methods used in \cite{Richter}, which relied on applications of the van der Corput inequality as a means of "complexity reduction and qualitative equidistribution results on nilmanifolds. In simplistic terms, this replaces the issue of studying equidistribution for a function $a(t)\in \mathcal{H}$ by the problem of studying equidistribution properties for the derivatives $a', a'', a'''$ and so on. This cannot be used to cover, for example, functions in the range $t\prec a(t)\ll t\log t$, because the derivative $a'$ must satisfy $a'(t)\ll \log t$, which does not have good equidistribution properties even on the 1-dimensional torus $\T$. As we mentioned above, we can sidestep this situation, but 
our argument also has limitations. More precisely, we do not cover functions that grow very slowly (which we call sub-fractional functions). A sample of a "slow-growing" function that we cannot handle is the function $\log^c t$ for $c>1$ (for instance, the pair $(\log^2t, t\log t)$ is not covered by Theorem \ref{equidistribution}). The main reason is that when we pass to averages on small intervals, these functions become approximately equal to a constant and our method of using the Taylor expansion breaks down. That explains the existence of the function $t^{\e}$ in \eqref{P1} instead of the term $\log t$, which is speculated to be optimal. 

In addition, we do not cover the case where some of the functions $a_i(t)$ are real polynomials, because the reduction to a statement on connected simply connected Lie groups becomes a lot more complicated. For example, consider a nilmanifold $X=G/\G$ with $G$ connected and simply connected and elements $b_1,b_2\in G$ that commute. It is not clear how to describe sufficiently well the orbit of the sequence $b_1^{n^{3/2}}b_2^{n^2}\G$ on $X$. On the one hand, invoking Leibman's theorem on polynomial orbits (Theorem \ref{Leibmantheorem}), we can describe the orbit $\overline{\{b^{n^2}\G\}}$, while the nilmanifold $\overline{b_1^{\R}\G}$ (which can be shown to be equal to the closure of the orbit $b^{n^{3/2}}\G$ by Theorem \ref{Fratheorem}) can be expressed in a nice form by Ratner's theorem (see also Lemma \ref{Ratnerlemma}). However, we do not know how to accomplish this for their product $b_1^{n^{3/2}}b_2^{n^2}\G$. For example, we expect that this sequence equidistributes on some subnilmanifold $Y$ (possibly after restricting to an arithmetic progression), but we cannot get any information on the underlying Lie group that defines $Y$, which is necessary when applying Theorem \ref{quantitative}.

A simple argument reduces our problem to the case when the Lie group $G$ is connected and simply connected. Namely, we will prove 
Theorem \ref{pointwise} under the above connectedness assumptions. We sketch this reduction in Appendix \ref{nilpotent crap} (at the end of subsection \ref{B1 subsection}).
Therefore, we make the following convention:
\begin{align}\label{convention}\tag{$\star$}
    \textit{For the rest of the article  }& \textit{up until the Appendix, we make the assumption that}\\
     \nonumber \textit{all nilpotent Lie groups}&  \textit{ are always connected and simply-connected}.
\end{align}

\subsection*{Notational conventions}
Throughout this article, we denote by $\N=\{1,2,...\}$ the set of natural numbers. We denote $\T^d=\R^d/\Z^{d}$, $e(t)=e^{2\pi it}$, while $\norm{x}_{\T}=d(x,\Z)$ and $\{x\}$ denote the distance of $x$ from the nearest integer and the fractional part of $x$ respectively. For an element ${\bf x}=(x_1,...,x_k)$ in $\R^k$, we denote $|{\bf x}|= |x_1|+\dots +|x_k|$. Lastly, we denote by $\bol_A$ the characteristic function of a set $A$.

For two sequences $a_n,b_n$, we say that $b_n$ dominates $a_n$ and write $a_n\prec b_n$ or $a_n=o(b_n)$, when $a_n/b_n$ goes to 0, as $n\to+\infty$. In addition, we write $a_n\ll b_n$ or $a_n=O(b_n)$, if there exists a positive constant $C$ such that $|a_n|\leq C|b_n|$ for large enough $n$. When we want to denote the dependence of this constant on some parameters $h_1,...,h_k$, we will use the notation $a_n=O_{h_1,...,h_k}(b_n)$. We use identical notation for asymptotic relations between functions on some real variable $t$, where we understand that these hold when we take $t\to+\infty$.

 Finally, we use the symbol $\E$ to denote averages (over a range that will be implicit by the corresponding subscripts each time). Throughout the rest of the article, we use the letters $p,q$ to denote polynomials and $\chi$ is used to denote a horizontal character. We will use $b_1,b_2,...,b_k$ or $u_1,u_2,...,u_k,w_1,...,w_k$ in the proofs to denote elements of a nilpotent Lie group $G$.

\subsection{Acknowledgements}
I would like to thank my PhD advisor Nikos Frantzikinakis for many helpful discussions. I would also like to thank the anonymous referee for pointing out corrections in the previous versions of the paper and for several additional valuable suggestions that improved the overall presentation of the article.

\section{Background material}\label{background}

\subsection{Measure preserving systems and Host-Kra structure theorem} A measure preserving system is a quadruple $(X,\X,\m,T)$, where $(X,\X,\m)$ is a Lebesgue probability space and $T$ is an invertible measure preserving map, that is $\m(T^{-1}(A))=\m(A)$ for all $A\in \X$. It is called ergodic if all the $T$-invariant functions are constant.
For the purposes of this article, a factor of the system $(X,\mathcal{X},\m, T)$ is a $T$-invariant sub-$\sigma$-algebra of $\mathcal{X}$. However, when there is no confusion, we will omit the $\sigma$-algebra $\mathcal{X}$ from the quadruple $(X,\mathcal{X},\m, T)$.

Let $(X,\m,T)$ be a measure preserving system and let $f\in L^{\infty}(\m)$. We define the Host-Kra uniformity seminorms inductively as follows:\begin{equation*}
    \nnorm{f}_{0,T}=\int f \ d\m
\end{equation*}and, for $s\in \N$, \begin{equation}\label{uniformitynorms}
    \nnorm{f}_{s+1,T}^{2^{s+1}}=\lim\limits_{H\to\infty}\ \underset{0\leq h\leq H}{\E} \nnorm{\bar{f}\cdot T^h f}_{s,T}^{2^s}.
\end{equation}
In the ergodic case, the existence of these limits and the fact that these quantities are indeed seminorms was established in \cite{Host-Kra1}. In the same article, it was shown that these seminorms give rise to a factor $\mathcal{Z}_{s-1}(X)$ of $X$ for every $s\geq 1$, which is characterized by the following relation: \begin{equation*}
    f\perp L^2(\mathcal{Z}_{s-1}(X))\iff \nnorm{f}_{s,T}=0.
\end{equation*}
The significance of these factors hinges on the following remarkable structure theorem of Host-Kra \cite{Host-Kra1}:

\begin{customthm}{E}[Host-Kra]
\label{structure}
Let $(X,\m,T)$ be an ergodic system. Then, the factor $\mathcal{Z}_s(X)$ is an inverse limit of $s$-step nilsystems.
\end{customthm}
The last property implies that there exist $T$-invariant sub-$\sigma$-algebras $\mathcal{Z}_s(n), n\in \N$ that span $\mathcal{Z}_s$, such that the factor $\mathcal{Z}_s(n)$ is isomorphic as a system to an $s$-step nilsystem.

\subsection{Background on Hardy fields}

Let $\mathcal{B}$ denote the set of germs at infinity of real-valued functions defined on a half-line $[x,+\infty)$. Then, $(\mathcal{B},+,\cdot)$ is a ring, and a sub-field $\mathcal{H}$ of $\mathcal{B}$ that is closed under differentiation is called a Hardy field. We will say that $a(n)$ is a Hardy sequence, if for $n\in \N$ large enough we have $a(n)=f(n)$ for a function $f\in\mathcal{H}$.

 Any two functions $f,g\in\mathcal{H}$ with $g$ not identically zero are comparable, that is the limit \begin{equation*}
    \lim\limits_{t\to\infty} \frac{f(t)}{g(t)}
\end{equation*}exists and thus it makes sense to compare their growth rates. In addition, every non-constant function in $\mathcal{H}$ is eventually monotone and, therefore, has a constant sign eventually. In Appendix \ref{hardyappendix}, we have collected some lemmas about growth rates of functions in $\mathcal{H}$, which will be used frequently throughout the proofs. The proofs of these lemmas can be found in \cite{tsinas}, so we shall omit most of them.

We define below some notions that will be used repeatedly throughout the remainder of the paper.
\begin{definition}\label{growthdefinitions}
Let $a$ be a function in $\mathcal{H}$.\\
a) The function $a$ has polynomial growth, if there exists a positive integer $k$ such that $f(t)\ll t^k$. The smallest positive integer $k$ for which this holds will be called the degree of $a$.\\
b) The function $a$ is called sub-linear if $a(t)\prec t$.\\
c) The function $a$ is called sub-fractional if $a(t)\prec t^{\e}$, for all $\e>0$.\\
d) The function $a$ is called strongly non-polynomial if, for any positive integer $k$, we have that the functions $a(t)$ and $t^k$ have distinct growth rates. 
\end{definition}
 If $a\in\mathcal{H}$ has polynomial growth, we will also say that the corresponding sequence $a(n)$ has polynomial growth throughout the article. To understand the definition, consider the functions $a_1(t)= t^{2/3}$, $a_2(t)=\log^2 t$, $a_3(t)= t+t^{1/2}$ and $a_4(t)=\exp(t)$. The first two functions are sub-linear, but the functions $a_3,a_4$ are not. The function $a_2(t)$ is the only sub-fractional function among the four functions (it grows slower than all fractional powers), while the strongly non-polynomial functions are $a_1,a_2$ and $a_4$ (note that $a_3$ grows like the polynomial $p(t)=t$). The function $a_4$ does not have polynomial growth.

\begin{remark}The definition of strongly non-polynomial presented here is slightly different than the one given in \cite{tsinas}. The definition in that article was that we have the growth relation $t^k \prec a(t)\prec t^{k+1}$ for $k\in \N$, which imposes polynomial growth on our function. In addition, our new definition also allows the inclusion of functions $a(t)$ such that $\lim\limits_{t\to+\infty}|a(t)|=0$, while the old one excludes these functions (we do this solely for technical reasons).
\end{remark}

%%%%%%%%%%%%%%%%%%%%%%%%%%%%%%%%%%%%%%%%%%%%%%%%%%%%%%%%%%%%%%%%%%%%%%%%%%%%%%%%%%%%%%%%%%%%%%%%%%%%%%%%%%%%%%%%%%%%%%%%%%%%%%%%%%%%%%%%%%%%%%%%%%%%%%%%%%%%%%%%%%%%%%%%%%%%%%%%%%%%%%%%%%%%%%%%%%%%%%%%%%%%%%%%%%%%%%%%%%%%%%%%%%%%%%%%%%%%%%%%%%%%%%%%%%%%%%%%%%%%%%%%%%%%%%%%%%%%%%%%%%%%%%%%%%%%%%%%%%%%%%%%%%%%%%%%%%%%%%%%%%%%

\section{Preparations for the proof}\label{preparations proof}

In this section, we will collect some lemmas and make some reductions, which will be useful when we delve into the proof of Theorems \ref{equidistribution} and \ref{pointwise} in the next section. In addition, we provide a specific example, which illustrates the central ideas of the proof of Theorem \ref{pointwise} and does not involve a lot of computations. 

First of all, we present a lemma, which appears in \cite[Lemma 3.3]{Fraeq}. We will use this lemma to reduce our problem of studying the long averages over an interval $[1, N]$ (like those appearing in Theorem \ref{pointwise}) to averages in short intervals. Its proof is elementary and so we omit it.

\begin{lemma}\label{short intervals}
 Let $(a(n))_{n\in \N}$ be a bounded sequence of complex numbers. Assume that \begin{equation*}
     \lim\limits_{N\to+\infty} \underset{N\leq n\leq N+L(N)}{\E} a(n) =  \alpha
 \end{equation*}for some positive function $L(t)$ with $1\prec L(t)\prec t$. Then, we also have \begin{equation*}
     \lim\limits_{N\to+\infty} \underset{1\leq n\leq N}{\E} a(n) =  \alpha.
 \end{equation*}
\end{lemma}

\subsection{An example of convergence}

 Assume $X=G/\G$ is a nilmanifold with $G$ connected and simply connected. We will show that the averages \begin{equation*}
    \underset{1\leq n\leq N}{\E} f(b_1^{n^{3/2}}x) \cdot g(b_2^{n\log n}x)
\end{equation*}converge for any $x\in X$, where $b_1,b_2\in G$. 

Using Lemma \ref{short intervals}, it suffices to show that the averages \begin{equation*}
     \underset{N\leq n\leq N+L(N)}{\E} f(b_1^{n^{3/2}}x) \cdot g(b_2^{n\log n}x) 
\end{equation*}converge, for some sub-linear function $L(t)$.
Passing to the nilmanifold $X\times X$, we see that our problem reduces to showing that the averages \begin{equation*}
     \underset{N\leq n\leq N+L(N)}{\E} F(b_1^{n^{3/2}} b_2^{n\log n}x)
\end{equation*}converge for any nilmanifold $X=G/\G$, commuting elements\footnote{When we pass to the product $X\times X$, we have to study the actions of the elements $(b_1,e_G)$ and $(e_G,b_2)$, which clearly commute.} $b_1,b_2\in G$ and function $F\in C(X)$. Due to density, we can actually pick $F\in \text{Lip}(X)$. We provide more details for this deduction in the next section (after Proposition \ref{impprop}).

Let $X'$ denote the subnilmanifold $\overline{b_1^{\R} b_2^{\R}\G}$ of $X$.
By Lemma \ref{Ratnerlemma}, this set is indeed a subnilmanifold of $X$ and has a representation as $H/\Delta$, with $H$ connected, simply connected and containing all elements $b_1^s$ and $b_2^s$ for any $s\in \R$. In this example, we will also assume that $X'=\overline{b_1^{\Z}b_2^{\Z}\G}$. In the main proof, we will use Lemma \ref{ergodicaction} to reduce the general case of the theorem to this one.

Using the Taylor expansion around the point $N$, we can write \begin{equation*}
    (N+h)^{3/2}=N^{3/2}+\frac{3}{2}hN^{1/2}+\frac{3h^2}{8N^{1/2}}-\frac{h^3}{16N^{3/2}}+\frac{3h^4}{128\xi_h^{5/2}}, \ \ \text{ for some } \ \ \xi_h\in [N,N+h]
\end{equation*}for every $0\leq h\leq L(N)$. If we choose $L(t)$ to satisfy \begin{equation*}
t^{1/2}    \prec L(t) \prec t^{5/8}
\end{equation*}then the last term in the above expansion is smaller than $o_N(1)$, while the second to last term is unbounded. Similarly, we can write \begin{equation*}
    (N+h)\log (N+h)=N\log N +h(\log N+1)+\frac{h^2}{2N}-\frac{h^3}{6\psi_h^2},  \ \ \text{ for some } \ \ \psi_h\in [N,N+h].
\end{equation*}If we choose again $L(t)$ to satisfy \begin{equation*}
t^{1/2}    \prec L(t) \prec t^{2/3}, 
\end{equation*}we can show that the last term is $o_N(1)$, while the $h^2$ term is unbounded. For instance, we can choose $L(t)=t^{3/5}$ and both growth conditions that we imposed will be satisfied.

Since the function $F$ is continuous, we can disregard the highest order terms in the above expansion since they are both $o_N(1)$. Our problem reduces to showing that the averages \begin{equation*}
    \underset{0\leq h\leq L(N)}{\E} F(b_1^{N^{3/2}+\frac{3}{2}hN^{1/2}+\frac{3h^2}{8N^{1/2}}-\frac{h^3}{16N^{3/2}}}b_2^{N\log N +h(\log N+1)+\frac{h^2}{2N}}x)
\end{equation*}converge. 
For the sake of simplicity, we will show that the averages \begin{equation*}
     \underset{0\leq h\leq L(N)}{\E} F(b_1^{\frac{h^3}{N^{3/2}}}b_2^{\frac{h^2}{N}}x)
\end{equation*}converge, since both of these statements follow from the same arguments. For convenience, we will also assume that $x=\G$.

Let $\delta>0$. We consider the finite sequence \begin{equation*}
    (v(h)\G)_{0\leq h\leq L(N)}= \Big(  b_1^{\frac{h^3}{N^{3/2}}}b_2^{\frac{h^2}{N}}\G       \Big)_{0\leq h\leq L(N)}
\end{equation*}and we show that, if $N$ is large enough, then it is $\delta$-equidistributed on the subnilmanifold $X'=\overline{b_1^{\R}b_2^{\R}\G}$ of $X$. It is apparent that $v(n)\G$ is a polynomial sequence in $X'$. 
We consider the horizontal torus $Z$ of $X'$, which is isomorphic to some $\T^{d}$ ($d\in \N$) and we also let $\pi$ denote the projection map from $X'$ to $ Z$.
If the given sequence is not $\delta$-equidistributed (for a fixed value of $N$), we can invoke Theorem \ref{quantitative} to find a positive constant $M=M(X',\delta) $ and a non-trivial horizontal character $\chi_N$ of modulus at most $M$ and such that \begin{equation*}
    \bignorm{\chi_N(\pi (v(h)\G))}_{C^{\infty}[L(N)]}\leq M.
\end{equation*}

Suppose $\chi_N$ descends to the character $$(t_1,...,t_d)\to e(k_{1,N}t_1+\dots +k_{d,N}t_d)$$ on $\T^d$, where $k_{1,N},...,k_{d,N}$ are integers. The fact that the modulus is bounded by $M$ implies that $$|k_{1,N}|+\dots+ |k_{d,N}|\leq M.$$ Let us also write $\pi(b_1\G)=(x_1,...,x_d)$ and $\pi(b_2\G)=(y_1,...,y_d)$. Then, the last inequality implies that \begin{equation}\label{exampleequation}
    \bignorm{ e\big(\frac{h^3}{N^{3/2}}\sum_{i=1}^{d} k_{i,N}x_i    +\frac{h^2}{N}\sum_{i=1}^{d}  k_{i,N} y_i \big)    }_{C^{\infty}[L(N)]}\leq M.
    \end{equation}

Assume there are infinitely many $N$ for which this holds.
Since there are only finitely many possible choices for the numbers $k_{1, N},...,k_{d, N}$ above, we conclude that there exists a character $\chi$ such that $\norm{\chi(\pi(a(h))\G )} _{C^{\infty}[L(N)]}\leq M$ holds for infinitely many $N\in \N$. Then, we rewrite \eqref{exampleequation} ($k_i$ are some integers independent of $N$) as \begin{equation*}
    \bignorm{ e\big(\frac{h^3}{N^{3/2}}\sum_{i=1}^{d} k_{i}x_i    +\frac{h^2}{N}\sum_{i=1}^{d}  k_{i} y_i \big)    }_{C^{\infty}[L(N)]}\leq M,
    \end{equation*}and this inequality holds for infinitely many $N$.

The definition of the $C^{\infty}[L(N)]$ norms implies that we have the relations \begin{equation*}
    L(N)^3\bignorm{ \frac{\sum_{i=1}^{d} k_ix_i  }{N^{3/2}} }_{\T}\leq M
\end{equation*}and 
\begin{equation*}
    L(N)^2\bignorm{\frac{\sum_{i=1}^{d}  k_i y_i  }{N}}_{\T}\leq M.
\end{equation*}for infinitely many $N$. Due to our choice of the function $L(N)$, these relations fail for $N$ sufficiently large unless \begin{equation*}
    \sum_{i=1}^{d} k_ix_i  \in \Z \ \ \text{ and }  \sum_{i=1}^{d}  k_i y_i  \in \Z.
\end{equation*}This implies that $\chi \circ \pi (b_1\G)=\chi\circ \pi(b_2\G)=0$ and, consequently, we must also have $\chi\circ \pi(b_1^mb_2^n\G)=0$ for any $m,n\in \Z$. Since elements of this form are dense in $\overline{b_1^{\R}b_2^{\R}\G}$ by our initial hypothesis, we get that $\chi$ must be the trivial character, which is a contradiction.

    In conclusion, we have established that the sequence $(v(h)\G)_{0\leq h\leq L(N)}$ is $\delta$-equidistributed for large enough $N$ on $X'=\overline{b_1^{\R}b_2^{\R}\G}$. The result now follows by sending $\delta\to 0$. We also notice that the limit of the averages is $\int_{X'} F \ d m_{X'}$.

\begin{remark}\label{diophantine}
We describe briefly here why we have to use the $t^{\e}$ term in \eqref{P1} instead of the conjectured optimal term $\log t$. Assuming we had the functions $\log^2 t$ and $t\log t$ in this example, then for any choice of the sub-linear function $L(t)$ that would give a good polynomial approximation for the function $t\log t$, we would have $$\max_{0\leq h\leq L(N)}|\log^2(N+h)-\log^2 N|=o_N(1),$$ which suggests that the sequence $\log^2 n$ is essentially constant in the small intervals $[N, N+L(N)]$. If we proceed exactly as in the above argument, the best we can actually show is that \begin{equation*}
    \big| \underset{N\leq n\leq N+ L(N)}{\E}F(b_1^{\log^2 n}b_2^{n\log n} \G)   -\int_{Y_2} F(b_1^{\log^2 N} y)\ dm_{Y_2}(y)\big|\leq \delta \norm{F(b_1^{\log^2 N }\cdot) }_{\text{Lip}(Y_2)}
\end{equation*}for large enough $N$, where $Y_2=\overline{b_2^{\R}\G}$ and $F(b_1^{\log^2 N }\cdot)$ denotes the function $y\to F(b_1^{\log^2 N }y)$ defined on the nilmanifold $Y_2$. However, the Lipschitz norm above is of the order $\log^2N \norm{F}_{\text{Lip}(X)}$, which diverges as $N\to+\infty$, so this bound cannot be useful for any purposes. 

Another approach would be to utilize the fact that the parameter $M$ in Theorem \ref{quantitative} is of the form $\delta^{-O(1)}$, namely we have bounds that are polynomial in $\delta$. Thus, one could allow the parameter $\delta$ to vary with $N$. For instance, establishing a bound of the form $(\log N)^{-(2+\e)} \norm{F(b_1^{\log^2 N }\cdot) }_{\text{Lip}(Y_2)}$ in place of the term $\delta \norm{F(b_1^{\log^2 N }\cdot) }_{\text{Lip}(Y_2)}$\footnote{It would actually suffice to obtain this statement for almost all $N\in \N$ in the sense of natural density.} (namely, showing that our sequence is $(\log N)^{-(2+\e)}$-equidistributed) on the right-hand side of the above equation leads to a solution to the more general problem. However, any bound of this type is incorrect in general. Indeed, assume that the horizontal torus of $\overline{b_2^{\R}\G}$ was $\T^2$ and also let $(b_{2,1},b_{2,2})\in \T^2$ denote the image of the element $b_2\G$ under the projection map. Following the same approximations as the ones in the example, we would like to show that the finite polynomial sequence $b_2^{h^2/N}\G$, where $0\leq h\leq L(N)$, is $ (\log N)^{-(2+\e)}$-equidistributed for almost all $N\in \N$ and for some suitable sub-linear function $L(t)$ satisfying only $L(t)\succ t^{1/2}$. Then, an application of Theorem \ref{quantitative} implies that if this assertion does not hold, then there exists a positive constant $C$ and a horizontal character $\chi$ of modulus at most $\log^C N$, such that 
$$\norm{\chi(b_2^{h^2/N}\G)}_{C^{\infty}(L(N))}\leq \log^C N.$$ Equivalently, there exist integers $k_1,k_2$ with $|k_1|+|k_2|\leq \log^C N$ such that
$$L^{2}(N)\bignorm{\frac{k_1b_{2,1}+k_2b_{2,2}}{N}}_{\T}\leq \log^C N.$$ Thus, we would get a contradiction if we showed that $$\min_{|k_1|,|k_2|\leq \log^C N} |k_1b_{2,1}+k_2b_{2,2}|\geq \frac{N\log^C N}{L^2(N)}$$ holds for $N$ in a set of density 1. However, we note that bounds like the above depend on the diophantine properties of the numbers $b_{2,1},b_{2,2}$. Indeed, let us suppose that $\alpha=\frac{b_{2,1}}{ b_{2,2}}\leq 1$. If we divide by $b_{2,2}$, the last inequality can be rewritten as $$\min_{|k_1|,|k_2|\leq \log^C N }\big|k_1\alpha+k_2\big|\geq \frac{N\log^C N}{|b_{2,2}|L^2(N)}.$$
For a fixed choice of $k_1$, the absolute value is minimized by picking $k_2$ to be the nearest integer to $-k_1\alpha.$ Thus, we would need to show that
$$\min_{ |k_1|\leq \log^C N }\bignorm{k_1\alpha}_{\T}\geq \frac{N\log^C N}{|b_{2,2}|L^2(N)}$$ and we can find $b_{2,1},b_{2,2}\in (0,1)$ for which this inequality fails for all $N$ in a set of positive upper density. A simpler example that avoids the complicated function on the right-hand side of the last equation is to show that
we can find $\alpha\in (0,1)$ for which the inequality $\min_{|k|\leq N}^{}\norm{k\alpha}_{\T}\geq 2^{-n}$ fails for all $N\in \N$ in a set of upper density 1. Indeed, we can construct an $\alpha\in (0,1)$ such that $\liminf\limits_{n\to+\infty}2^{2^n}\norm{n\alpha}_{\T}=0$. Thus, there is a sequence $q_n$ such that $\norm{q_na}_{\T}\leq 2^{-2^{q_n}}$ which implies that $\min_{|k|\leq N}^{}\norm{k\alpha}_{\T}\leq 2^{-2^{q_n}}\leq 2^{N}$ for every $N$ with $q_n\leq N\leq 2^{q_n}$. Thus, the set of $N$ for which the above inequality fails has upper density 1. 
\end{remark}

\subsection{Removing the integer parts}

In this part, we will establish a lemma that practically implies that part a) of Theorem \ref{pointwise} follows from part b) of the same theorem. The fact that part a) of Theorem \ref{equidistribution} follows from part b) of the same theorem is precisely the statement of \cite[Lemma 5.1]{Fraeq}, which is proven using very similar arguments to the proof of Lemma \ref{removerounding} below. If a collection of sequences of real numbers has the property that the averages \begin{equation}\label{pointwisegood}
    \underset{1\leq n\leq N}{\E}\  f_1(b_1^{{a_1(n)}}x_1)\cdot...\cdot f_k( b_k^{{a_k(n)}}x_k)  
\end{equation}converge for all nilmanifolds $X_i=G_i/\G_i$, elements $b_i\in G_i$, points $x_i\in X_i$ and continuous functions $f_i$ defined on $X_i$, we will say that this collection is pointwise good for nilsystems. The notation $b_i^{a_i(n)}$ makes sense here due to the connectedness assumptions we have imposed on the Lie groups $G_i$.

\begin{lemma}\label{removerounding}
Let $a_1(n),...,a_k(n)$ be sequences of real numbers that satisfy the following: \\
    a) The collection $a_1(n),...,a_k(n)$ is pointwise good for nilsystems.\\
    b) For every $1\leq i\leq k$, we have that the sequence $(a_i(n)\Z)_{n\in\N}$ satisfies one of the following:
    \begin{enumerate}
        \item It is equidistributed on $\T$.
        \item It converges to some $c=c(i)\in \T$ different from $0$.
        \item It converges to 0 and the sequence $\{a_i(n)\}-\frac{1}{2}$ has a constant sign eventually.
    \end{enumerate}

Then, the sequences $\floor{a_1(n)},...,\floor{a_k(n)}$ are pointwise good for nilsystems.
\end{lemma}
\begin{remark}
The number $\frac{1}{2}$ in the third condition is arbitrary since we could have used any number $\alpha\in (0,1)$. We primarily use this condition in the following manner: suppose we have a function $f(t)$, which converges monotonically to some $k\in \Z$ as $t\to+\infty$. Then, we clearly have $\norm{f(t)}_{\T}\to 0$ and we also observe that the sequence $\{f(n)\}$ does not not oscillate between intervals of the form $[0,\e]$ and $[1-\e,1)$ (due to the monotonicity assumption). Thus, the sequence $\{f(n)\}-\frac{1}{2}$ will indeed have a constant sign (positive if $f$ increases to $k$ and negative otherwise).
\end{remark}
\begin{proof}

Let $X_i=G_i/\G_i$ be nilmanifolds with $G_i$ connected and simply connected and $b_i\in G_i$. Let $f_1,...,f_k$ be continuous functions defined on $X_1,...,X_k$ respectively. Under the hypotheses of the lemma on the sequences $a_1(n),...,a_k(n)$, we want to show that the averages \begin{equation}\label{floorremove}
     \underset{1\leq n\leq N}{\E}\ f_1(b_1^{\floor{a_1(n)}}x_1)\cdot...\cdot f_k( b_k^{\floor{a_k(n)}}x_k)  
\end{equation}converge for any choice of the $x_i\in X_i$. 

Fix some $i\in\{1,2,...,k\}$. If the sequence $a_i(n)$ satisfies the second condition, namely that $a_i(n)\Z$ converges to $c\Z$ ($c\neq 0$), then, for $n$ sufficiently large, we have $$\floor{a_i(n)}=a_i(n)-\{c\}+o_n(1).$$ This implies that $b_i^{\floor{a_i(n)}}=b_i^{-\{c\}}b_i^{a_i(n)+o_n(1)}$. Since the function $f_i$ is continuous, we can disregard the contribution of the $o_n(1)$ term, while the $b_i^{-\{c\}}$ term can be absorbed by the $x_i$. Therefore, we notice that in this case, we can remove the integer part for the sequence $a_i(n)$.
An entirely similar argument demonstrates that the same holds if $a_i(n)$ satisfies the third condition.

In order to complete the proof, we will consider below the case that each of the sequences $a_i(n)\Z$ is equidistributed on $\T$ for convenience (namely, they all satisfy the first condition). Since we can easily remove the integer parts for those sequences that satisfy the second or third condition as we did above,
the argument below easily adapts to the general setting with some changes in notation.

We rewrite the averages in \eqref{floorremove} as \begin{equation*}
     \underset{1\leq n\leq N}{\E} \prod_{i=1}^k {f_i}(b_i^{-\{a_i(n)\}}b_i^{{a_i(n)}}x_i)=
      \underset{1\leq n\leq N}{\E} \prod_{i=1}^k \wt{f_i}(a_i(n)\Z,b_i^{a_i(n)}x_i)
\end{equation*}where $\wt{f_i}:\T\times X_i\to\C$ is  the function defined by the relation \begin{equation*}
    \wt{f_i}(s\Z,gx)=f_i(b_i^{-\{s\}}gx).
\end{equation*} 

Let $v_i(n)$ be the sequence $(a_i(n)\Z,b_i^{a_i(n)}x_i)$. By our hypothesis, for any continuous functions $f_i'$ on $\wt{X_i}=\T\times X_i$, the averages of  $\prod_{i=1}^{k}f_i'(v_i(n))$ converge. However, note that the functions $\wt{f_i}$ that we are dealing with may have discontinuities when $s$ becomes close to an integer. Our goal is to approximate each $\wt{f_i}$ by a continuous function and then use the above observation.

Let $\e>0$. For every $1\leq i\leq k$, we define a continuous function $f_{i,\e}$ that agrees everywhere with $\wt{f_i}$ on $[\e,1-\e]\times X_i$ and such that $f_{i,\e}$ is bounded uniformly by $2\norm{\wt{f_i}}_{\infty}$. Observe that \begin{multline}\label{2ke}
    \big| \underset{1\leq n\leq N}{\E} \wt{f_i}(v_i(n))- \underset{1\leq n\leq N}{\E} {f_{i,\e}}(v_i(n)) \big| =\frac{1}{N}\big|\sum_{\underset{ a_i(n)\notin [\e,1-\e]}{1\leq n\leq N}   }^{}\big(\wt{f_i}(v_i(n))-{f_{i,\e}}(v_i(n))\big) \big|\ll \\
    \e\norm{\wt{f_i}}_{\infty}+o_N(1)
\end{multline}where the last bound follows from the triangle inequality and the fact that $a_i(n)$ is equidistributed $(mod\ 1)$, which indicates that the set $\{n\in \N\colon a_i(n)\notin [\e,1-\e]\}$ has asymptotic density $2\e$.

Combining \eqref{2ke} with a simple telescoping argument, we deduce that \begin{equation*}
   \limsup_{N\to+\infty} \big| \underset{1\leq n\leq N}{\E} \prod_{i=1}^k \wt{f_i}(v_i(n)) - \underset{1\leq n\leq N}{\E} \prod_{i=1}^{k}{f_{i,\e}}(v_i(n))      \big|\ll k\e\prod_{i=1}^k \norm{\wt{f_i}}_{\infty}.
\end{equation*}
Since the averages $\underset{1\leq n\leq N}{\E} \prod_{i=1}^k{f_{i,\e}}(v(n))  $ converge as $N\to \infty$ by our hypothesis (the functions involved here are continuous), we infer that the averages $$\underset{1\leq n\leq N}{\E} \prod_{i=1}^k \wt{f_i}(v_i(n))$$ form a Cauchy sequence and, therefore, converge. The conclusion follows.
\end{proof}

Using the previous lemma, we can establish that the first part of Theorem \ref{pointwise} follows from the second part. We postpone this until the next section, where we also prove the second part of Theorem \ref{pointwise}.

\section{Proofs of main theorems}\label{PROOFS}

The main tool we are going to utilize in our proof is the quantitative Green-Tao theorem on polynomial orbits (Theorem \ref{quantitative}). A technical obstruction in our proof is that among the functions $a_1,...,a_k$ in the statement of Theorem \ref{pointwise}, we must separate the polynomial functions from the strongly non-polynomial ones. We will accomplish this using an elementary lemma (Lemma \ref{growthdecomposition}) which is proven in the Appendix. We restate Theorem \ref{pointwise} here:

\begin{customthm}{1.2}
Let $\mathcal{H}$ be a Hardy field that contains the polynomial functions. Let $a_1,...,a_k$ be functions in $\mathcal{H}$ that have polynomial growth. Assume that there exists $\e>0$, such that every function $a\in \mathcal{L}(a_1,...,a_k)$ satisfies either
\begin{equation*}
    \lim\limits_{t\to+\infty} \frac{|a(t) -p(t)   |}{t^{\e}}=+\infty    \ \text{ for any polynomial }\ p(t)\in \Q[t],
\end{equation*} or \begin{equation*}
  \text{the limit }\  \lim\limits_{t\to+\infty} a(t)  \ \text{is a real number}.
    \end{equation*}Then, we have the following:\\
    (i) For any collection of nilmanifolds $X_i=G_i/\Gamma_i$, elements $b_i\in G_i$, $x_i\in X_i$ and continuous functions $f_1,...,f_k$ with complex values, the averages \begin{equation*}
        \frac{1}{N}\sum_{i=1}^{N} f_1(b_1^{\floor{a_1(n)}}x_1)\cdot \dots \cdot f_k(b_k^{\floor{a_k(n)}}x_k)
    \end{equation*}converge.\\
(ii) For any collection of nilmanifolds $X_i=G_i/\Gamma_i$ such that the groups $G_i$ are connected, simply connected, elements $b_i,\in G_i$, $x_i\in X_i$ and continuous functions $f_1,...,f_k$ with complex values, the averages \begin{equation*}
        \frac{1}{N}\sum_{i=1}^{N} f_1(b_1^{a_1(n)}x_1)\cdot... \cdot f_k(b_k^{a_k(n)}x_k)
    \end{equation*}converge.
\end{customthm}

First of all, we show that the first part follows from the second part. This is accomplished by using Lemma \ref{removerounding}. We remark again that in part i), there are no connectedness assumptions made on the groups $G_i$. Nonetheless, the convention \eqref{convention} in Section \ref{introduction} allows us to consider only the case that the Lie groups $G_i$ are connected and simply connected. We implicitly work under this assumption in the proof below.

\begin{proof}[Proof of part i) of Theorem \ref{pointwise}, assuming part ii)]

We will have to confirm that the conditions of Lemma \ref{removerounding} are satisfied. Let $a_1,...,a_k\in \mathcal{H}$ be as in the statement of Theorem \ref{pointwise}. Condition a) of Lemma \ref{removerounding} is satisfied by our hypothesis. Now, we verify the second condition. 

Fix some $i\in \{1,2,..,k\}$. We consider three cases:\\
i) Assume that the function $a_i(t)$ is such that $|a_i(t)-q(t)|\succ t^{\e}$ for all polynomials $q(t)$ with rational coefficients. Then, the sequence $a_i(n)\Z$ is equidistributed on $\T$ (satisfying condition (1)), due to Theorem \ref{Bosh}.\\
ii) Assume that the function $a_i(t)$ is such that $\lim\limits_{t\to+\infty} a_i(t)=c\notin \Z$. Then, the sequence $a_i(n)$ satisfies condition (2) of Lemma \ref{removerounding}.\\
iii) Assume that neither of the above conditions is true. Since $a_i(t)$ must satisfy \eqref{P2}, we deduce that $a_i(t)$ converges to some integer $c$. However, since $a_i(t)$ converges to $c$ monotonically (functions in $\mathcal{H}$ are eventually monotone), we deduce that condition (3) of Lemma \ref{removerounding} is satisfied and we are done.
\end{proof}

Now we switch our attention to the proof of part ii). Firstly, we will apply Lemma \ref{growthdecomposition} from Appendix \ref{hardyappendix} in order
to replace the original functions $a_1,...,a_k$ with a collection of functions that are more manageable. This will enable us to separate the polynomial functions from strongly non-polynomial ones. In addition, among the strongly non-polynomial functions, we have to isolate those that are sub-fractional, because they behave differently when we try to employ the Taylor expansion. This whole process will reduce Proposition \ref{impprop} below to Lemma \ref{Lemma before step 2}, which we will then proceed to establish.

Following all these reductions, we use the Taylor expansion to substitute the strongly non-polynomial functions with polynomials in some small intervals. Now, this reduces the original problem to a quantitative equidistribution problem of finite polynomial sequences in a nilmanifold, although the coefficients of the polynomials vary depending on the underlying short interval. Finally, in Step 3, we use the quantitative equidistribution results to show that averages of Lipschitz functions in the nilmanifold over these "variable" polynomial sequences are very close to an integral over a subnilmanifold, which ultimately allows us to evaluate the limit of the initial averages.

We make one final reduction:
let $a_1,...,a_k\in\mathcal{H}$ be functions as in the statement of Theorem \ref{pointwise}. Passing to the product nilmanifold, we infer that our problem follows from the following statement:
\begin{proposition}\label{impprop}
Let $X=G/\G$ be a nilmanifold, $b_1,...,b_k\in G$ are commuting elements and $a_1,...,a_k\in \mathcal{H}$ have polynomial growth. Assume that there exists $\e>0$, such that every function $a\in \mathcal{L}(a_1,...,a_k)$ satisfies either \eqref{P1} or \eqref{P2}. Then, for any $x\in X$ and continuous function $F:X\to \C$, we have that the averages \begin{equation}\label{sht}
    \underset{1\leq n\leq N}{\E} F(b_1^{a_1(n)}\dots b_k^{a_k(n)} x)
\end{equation}converge. 

\end{proposition}
\begin{proof}[Proof that Proposition \ref{impprop} implies Theorem \ref{pointwise}]
We want to show that the averages \begin{equation*}
        \frac{1}{N}\sum_{i=1}^{N} f_1(b_1^{a_1(n)}x_1)\cdot ... \cdot f_k(b_k^{a_k(n)}x_k)
    \end{equation*}converge for all $x_i\in X_i$, where the nilmanifolds $X_i=G_i/\G_i$, the elements $b_i$ and the functions $a_i\in \mathcal{H}$ are as in the statement of part (ii) of Theorem \ref{pointwise}. We define the continuous function $F$ on the product nilmanifold $X_1\times \dots \times X_k$ by the relation $$F(y_1,...,y_k)=f_1(y_1)\cdot...\cdot f_k(y_k).$$ We also denote by $\wt{b_i}$ the element on $G_1\times \dots \times G_k$, whose $i$-th coordinate is equal to $b_i$, while all of its other coordinates are equal to the respective identity element. Observe that the elements $\wt{b_1},...,\wt{b_k}$ are pairwise commuting. Finally, let us also denote by $x$ the point $(x_1,...,x_k)$ on the product $X_1\times\dots\times X_k$. Then, a simple computation implies that our initial average is equal to
        \begin{equation*}
        \underset{1\leq n\leq N}{\E} F(b_1^{a_1(n)}\dots b_k^{a_k(n)} x)
    \end{equation*}and the claim now follows.
\end{proof}

Now, we will reduce Proposition \ref{impprop} to the following lemma:

\begin{lemma}\label{Lemma before step 2}
    Let $G/\G$ be a nilmanifold and suppose that $u_1,...,u_{s}$ are elements in $G$, such that  \begin{equation}\label{R to Z action}
        \overline{u_1^{\R}...u_{s}^{\R}\G}=\overline{u_1^{\Z}...u_{s}^{\Z}\G}.
    \end{equation} In addition, assume that the nilmanifold $X'=\overline{u_1^{\R}...u_{s}^{\R}\G}$ can be represented as $G'/\G'$, where $G'$ is connected, simply connected and contains all elements $u_1,...,u_{s}$. Let $s_0,s$ be positive integers and define the sequence $v(n)$\begin{equation}\label{v2}
    \prod_{i=1}^{s_0} u_i^{p_i(n)+x_i(n)} \prod_{i=s_0+1}^{s}u_i^{\wt{p}_i(n)+x_i(n)},
\end{equation}where:\\
a) $p_i,\wt{p}_j$ are polynomials with real coefficients, such that every non-trivial linear combination of the polynomials $\wt{p}_{s_0+1},...,\wt{p}_s$ is not an integer polynomial,\\
b) the functions $x_{i}$ are all strongly non-polynomial, the functions $x_1,...,x_{s_0}$ are not sub-fractional and have pairwise distinct growth rates and the functions $x_{s_0+1},...,x_s$ are sub-fractional.\\
Then, for any Lipschitz function $F$ on $X'$ with Lipschitz norm at most 1, the averages \begin{equation*}
    \underset{1\leq n\leq N}{\E} F(\prod_{i=1}^{s_0} u_i^{p_i(n)+x_i(n)} \prod_{i=s_0+1}^{s}u_i^{\wt{p}_i(n)+x_i(n)}\G')
\end{equation*}converge to the integral $\int _{X'}F \ dm_{X'}$.
\end{lemma}

While the statement may seem relatively convoluted at first, the sequence $v(n)$ above has a convenient form, so that the Taylor approximation can be used directly. 

First of all, we prove that Lemma \ref{Lemma before step 2} implies Proposition \ref{impprop}. We will rely on Lemma \ref{growthdecomposition}
to make the required reductions on the Hardy field functions in the iterates and we will also use Lemma \ref{ergodicaction} to get the equality \eqref{R to Z action}, where $u_1,\dots,u_s$ will be some appropriate elements of the Lie group $G$ (they will be products of powers of the elements $b_i$ in Proposition \ref{impprop}).

\begin{proof}[Proof that Lemma \ref{Lemma before step 2} implies Proposition \ref{impprop}]

Applying Lemma \ref{growthdecomposition}, we can find a basis ${f_1,...,f_s}$ for the set $\mathcal{L}(a_1,...,a_k)$ of non-trivial linear combinations. The collection of functions $f_1,...,f_s$ can be written in the form $(g_1,,,.,g_m,h_1,...,h_{\ell})$ where $g_i,h_{i}$ are as in Lemma \ref{growthdecomposition}. We will not use this specific property until a little further below, so as to avoid cumbersome notation. Note that the fact that $f_1,...,f_s$ form a basis indicates that the assumptions on the linear combinations of the $a_1,...,a_k$ in the statement of Proposition \ref{impprop} are now transferred to the functions $f_1,...,f_s$.

If we write \begin{equation}\label{linearexpansion}
    a_i(t)=\sum_{j=1}^{s} c_{i,j}f_j(t),
\end{equation}for some real numbers $c_{i,j}$, then we can rewrite the average in \eqref{sht} as \begin{equation}\label{Fu1}
    \underset{1\leq n\leq N}{\E} F(u_1^{f_1(n)}\dots u_s^{f_s(n)}x)
\end{equation}for some commuting elements $u_1,...,u_s\in G$ (here, the fact that the elements $b_1,...,b_k$ commute is required).
We denote $$v(n)=u_1^{f_1(n)}\dots u_s^{f_s(n)},$$which is a sequence in $G$. We want to establish that the averages of the sequence $F(v(n)x)$ converge for all $x\in X$ and any continuous function $F$.
If one of the functions $f_1,...,f_m$ is such that the limit $\lim\limits_{t\to+\infty} f_i(t)$ is a real number (which can be the case when a linear combination of the original functions satisfies \eqref{P2}), we can invoke the continuity of $F$ to eliminate the corresponding term $u_i^{f_i(n)}$ in the product and replace it by a constant. Hence, we may assume that all of the functions $f_1(t),...,f_s(t)$ go to $\pm \infty$, as $t\to+\infty$.

Now we use the particular structure of the functions $f_1,...,f_s$.
The statement of Lemma \ref{growthdecomposition} implies that the collection of functions $f_1,...,f_s$ has the form $(g_1,...,g_m,h_1,...,h_{\ell})$ (clearly, $m+\ell=s$) such that the functions $g_i$ can be written in the form $p_i(t)+x_i(t)$, where the functions $x_1(t),...,x_{m}(t)$ are strongly non-polynomial and have pairwise distinct (and non-trivial) growth rates, while the functions $h_i$ can be written in the form $\wt{p}_i(t)+y_i(t)$, where $y_i(t)$ converges to 0. Here, $p_i$ and $\wt{p}_i$ are polynomials with real coefficients.

We may rearrange the functions $f_i$ so that $f_i=g_i$ for all $1\leq i\leq m$ and $f_j=h_{j-m}$ for each $m+1\leq j\leq s$. Rewrite the sequence $v(n)$ as \begin{equation*}
    v(n)=\prod_{i=1}^{m}u_i^{g_i(n)}\cdot \prod_{i=1}^{\ell} u_{m+i}^{h_i(n)}=\prod_{i=1}^{m}u_i^{p_i(n)+x_i(n)}\cdot \prod_{i=1}^{\ell} w_i^{\wt{p}_i(n)+y_i(n)},
\end{equation*}where we use the notation $w_i$ for the element $u_{i+m}$ in the last equality. Without loss of generality, assume that $$x_1(t)\succ  x_2(t)\succ \cdots \succ x_{m}(t)\succ 1.$$

Firstly, we need to distinguish between the sub-fractional functions and the "fast" growing functions among the functions $x_i(t)$ (this will be important later when we use the polynomial expansion). Thus, let $0\leq s_0\leq m$ be a natural number such that $x_{s_0}(t)\gg t^{\e}$ for some $\e>0$, while $x_{s_0+1}$ is a sub-fractional function. This also implies that all the functions $x_{i}$ for $i$ satisfying $\ s_0+1\leq i\leq m$ are sub-fractional since we have arranged the functions so that their growth rates are in descending order. 

Once again, we rewrite the sequence $v(n)$ in the form \begin{equation*}
    v(n)= \prod_{i=1}^{s_0} u_i^{p_i(n)+x_i(n)} \prod_{i=s_0+1}^{m} u_i^{p_i(n)+x_i(n)}  \  \prod_{i=1}^{\ell}\  w_i^{\wt{p}_i(n)+y_i(n)} .
\end{equation*}Because the function $F$ is continuous, we can discard the functions $y_1,...,y_{\ell}$, since they all converge to zero. The hypotheses \eqref{P1} and \eqref{P2} on the linear combinations of the remaining functions in the exponents continue to hold. Indeed, this can be seen by noting that \eqref{P1} and \eqref{P2} still hold when replacing one of the functions (say $a_1$) by a function of the form $a_1(t)+e(t)$, with $e(t)\to 0$. Consequently, we can redefine $v(n)$ to be the sequence \begin{equation*}
     v(n)= \prod_{i=1}^{s_0} u_i^{p_i(n)+x_i(n)} \prod_{i=s_0+1}^{m} u_i^{p_i(n)+x_i(n)}  \  \prod_{i=1}^{\ell}\  w_i^{\wt{p}_i(n)} .
\end{equation*}

We will now reduce our problem to the case that the polynomials $\wt{p}_1(t),...,\wt{p}_{\ell}(t)$ are linearly independent.
Due to our hypothesis (namely \eqref{P1},\eqref{P2}), every non-trivial linear combination of the functions $\wt{p}_1(t),...,\wt{p}_{\ell}(t)$ must satisfy either \eqref{P1} or \eqref{P2}. Thus, every linear combination of the polynomials $\wt{p}_1(t),...,\wt{p}_{\ell}(t)$ is not a polynomial with integer coefficients unless it is the zero polynomial. If the second case is true, there exist $c_1,...,c_{\ell-1}\in \R$ such that \begin{equation*}
    \wt{p}_{\ell}=c_1\wt{p}_1+\dots+ c_{\ell -1}\wt{p}_{\ell-1}.
\end{equation*}Then, we have 
$$\prod_{i=1}^{\ell} w_{i}^{\wt{p}_i(n)}=\prod_{i=1}^{\ell-1}(w_{i}w_{\ell}^{c_i})^{\wt{p}_i(n)}. $$
If the polynomials $\wt{p}_1,...,\wt{p}_{\ell-1}$ are linearly independent, then we are done. Otherwise, we proceed similarly to eliminate $\wt{p}_{\ell-1}$. After a finite number of steps, we will reach a collection of linearly independent polynomials.

In view of the above, we are allowed to assume that $\wt{p}_1,...,\wt{p}_{\ell}$ are linearly independent. 
Now, we show that we can reduce to the case that the polynomials $p_{s_0+1},\dots, p_{m},\wt{p}_{1},\dots, \wt{p}_{\ell}$.
Indeed, the linear independence assumption on the polynomials $\wt{p}_1,\dots, \wt{p}_{\ell}$ implies that the polynomials $p_{s_0+1},...,p_m,\wt{p}_1,...,\wt{p}_{\ell}$ are linearly independent. To see how this works, observe that if there are real numbers $c_i,d_j$ such that $$\sum_{i=1}^{m-s_0}c_ip_{s_0+i}+\sum_{i=1}^{\ell}d_ip'_i=0,$$ then the function \begin{equation*}
    \sum_{i=1}^{m-s_0}c_i(p_{s_0+i}+x_{s_0+i})+\sum_{i=1}^{\ell}d_i\wt{p}_i= \sum_{i=1}^{m-s_0}c_ix_{s_0+i}
\end{equation*}is a sub-fractional function that does not converge to 0, since the functions $x_{s_0+i}$ are sub-fractional and have pairwise distinct growth rates. This contradicts our hypothesis (specifically \eqref{P1}) and our claim follows.

In conclusion, we see that the sequence $v(n)$ can be written in the form \begin{equation}
    \prod_{i=1}^{s_0} u_i^{p_i(n)+x_i(n)} \prod_{i=s_0+1}^{m}u_i^{p_i(n)+x_i(n)} \prod_{i=1}^{\ell}w_i^{\wt{p}_i(n)},
\end{equation}
where the functions $x_{i}$ are strongly non-polynomial with distinct growth rates, the functions $x_1,...,x_{s_0}$ are not sub-fractional, the functions $x_{s_0+1},...,x_s$ are sub-fractional and every non-trivial linear combination of the polynomials $p_{s_0+1},...,p_m,\wt{p}_1,...,\wt{p}_{\ell}$ is not an integer polynomial. We also recall that we have arranged the functions $x_i$ to be in decreasing order with respect to their growth rates.

We can combine the last two factors of this product into one factor to simplify our problem a bit more. More specifically, we can rewrite the sequence $v(n)$ in the form (we make some mild modifications in our notation here) 
\begin{equation}\label{v(n)}
    v(n)=\prod_{i=1}^{s_0} u_i^{p_i(n)+x_i(n)} \prod_{i=s_0+1}^{s}u_i^{\wt{p}_i(n)+x_i(n)},
\end{equation}where $s=m+l$, $p_i,\wt{p}_j$ are real polynomials, the functions $x_{i}$ are strongly non-polynomial with distinct growth rates, $x_1,...,x_{s_0}$ are not sub-fractional, $x_{s_0+1},...,x_s$ are sub-fractional and every non-trivial linear combination of the polynomials $\wt{p}_{i}$ is not an integer polynomial. Namely, our functions satisfy hypotheses a) and b) of Lemma \ref{Lemma before step 2}.

In order to establish our assertion, it suffices to show that the sequence $v(n)x$ (where $v(n)$ is as in \eqref{v(n)}) is equidistributed on the nilmanifold $X'=\overline{u_1^{\R}\dots u_s^{\R}x}$ for any $x\in X$. We will prove this in the case $x=\G$ since the general case follows from this using the change of base point trick that we discuss in Appendix \ref{nilpotent crap} (see Sub-subsection \ref{change of base point}). In addition, we can invoke Lemma \ref{ergodicaction} to find a real number $s_0$, such that $X'=\overline{(u_1^{s_0})^{\Z}...(u_s^{s_0})^{\Z}\G}$. Replacing the functions $p_i(t)+x_i(t)$ ($1\leq i\leq s_0$) by the functions $\big(p_i(t)+x_i(t)\big)/s_0$ and $\wt{p}_i(t)+x_i(t)$ ($s_0+1\leq i\leq s$) by $\big(\wt{p}_i(t)+x_i(t)\big)/s_0$ (the assumptions on the linear combinations of the functions remain unaffected), we can reduce our problem to the case that $X'=\overline{u_1^{\Z}...u_s^{\Z}\G}$.

We want to show that for any continuous function $F$ from $X'=G'/\G'$ ($G'$ is connected, simply connected and $\G'$ is a uniform subgroup), the averages \begin{equation*}
    \underset{1\leq n\leq N}{\E} F(v(n)\G')
\end{equation*}converge to the integral $\int _{X'}F \ dm_{X'}$.
Since Lipschitz functions are dense in the space $C(X')$, we may assume that $F$ is Lipschitz continuous. In addition, we may assume after rescaling that $\norm{F}_{\text{Lip}(X')}\leq 1$. Now, our claim follows immediately from Lemma \ref{Lemma before step 2}.
\end{proof}

In the following part, we will prove Lemma \ref{Lemma before step 2}. We divide the proof into two steps. During Step 1, we will approximate the functions $x_1,\dots,x_s$ by polynomials in a suitable short interval. Our goal is to reach an average over a short interval of the form $[N, N+L(N)]$ of a sequence of the form $F(g(n)x)$, where $F$ is Lipschitz and $g(n)$ is a polynomial sequence on the nilmanifold $X'$ (the polynomial sequence will vary with the parameter $N$).
This will be ensured by Proposition \ref{common Taylor expansion}. In step 2, we will use Theorem \ref{quantitative} to deduce that these averages are close to the integral of $F$ for large values of $N$.

All the reductions above allow us to write $v(n)$ in a form that will be appropriate for the application of the quantitative equidistribution theorem (after we perform the Taylor expansion).
When we apply the Taylor expansion in the first step, the functions $x_{s_0+1},...,x_{s}$ will become approximately constant and thus the desired equidistribution will be mainly "affected" by the polynomials $\wt{p}_{s_0+1},...,\wt{p}_s$. On the other hand, the functions $x_1,...,x_{s_0}$ will play a meaningful role in the equidistribution of our sequence. In particular, the presence of the functions $x_1,\dots,x_{s_0}$ will imply "closeness" of our averages to the integral of the Lipschitz function $F$, unless the projections of the elements $u_1,\dots, u_{s_0}$ on the horizontal torus are zero. In this second case, condition $a)$ on the polynomials completes the proof.
Lastly, the "linear independence" condition of the polynomials $\wt{p}_{s_0+1},...,\wt{p}_s$ guarantees that the projection of the sequence $v(n)$ on $X'$ will be equidistributed on the entire nilmanifold $\overline{u_1^{\R}... u_{s}^{\R}\G}$, since, otherwise, we would need to pass to some subnilmanifold to guarantee equidistribution (and to an appropriate arithmetic progression).

\begin{proof}[Proof of Lemma \ref{Lemma before step 2}]

{\bf Step 1: Approximating by polynomials:}
Let $L(t)$ be a sub-linear function with $\lim\limits_{t\to+\infty}L(t)=+\infty$ that we will determine later. It suffices to show that the sequence of the averages \begin{equation}\label{Fu2}
    \underset{N\leq n\leq N+L(N)}{\E} F(v(n)\G') 
\end{equation}converges to $\int _{X'}F \ d m_{X'}$, since the conclusion would follow from Lemma \ref{short intervals}. Reordering if necessary, we assume again that \begin{equation*}
    x_1(t)\succ\dots \succ x_{s_0}(t).
\end{equation*}

Let $r$ be a very large natural number compared to the degrees of the polynomials $p_i,\wt{p}_j$ and the degrees of the functions $x_i(t)$. If $r$ is sufficiently large, we have that $x_i^{(r)}(t)=o_t(1)$ for all $i\in \{1,\dots,s_0\}$. Assuming again that $r$ is sufficiently large, then for any function $L(t)$ that satisfies \begin{equation*}
 (x_i^{(r)}(t))^{-1/r}   \prec L(t)\prec t^{1-\e'}
\end{equation*}for some $\e'>0$ and all $i\in{1,\dots ,s_0}$, we have that for each $ i\in\{1,...,s_0\}$, there is a unique natural number $k_i\geq r$ so that the sub-class $S(x_i,k_i)$ contains the function $L(t)$ (this follows from Proposition \ref{basic}). The fact that the function $L(t)$ belongs to $S(x_i,k_i)$ indicates that we have the relations \begin{equation}\label{growthinequalitiessht}
    (x_i^{(k_i)}(t))^{-1/k_i} \prec L(t)\prec (x_i^{(k_i+1)}(t))^{-1/(k_i+1)}.
\end{equation}

We can guarantee that the numbers $k_i$ are also very large compared to the degrees of the polynomials $p_{j},\wt{p}_{j'}$ by enlarging the number $r$ in the beginning\footnote{For example, assuming that $k_i$ is at least 10 times as large as the maximal degree appearing among the polynomials $p_i,\wt{p}_j$ and 10 times as large as the number $s$ of all existing polynomials would suffice for our arguments.}.

We use the Taylor expansion for the functions $x_1(t),...,x_{s_0}(t)$ to write \begin{equation}\label{Taylorexp}
    x_i(N+h)=x_i(N)+\dots +\frac{x_{i}^{k_i}(N)h^k}{k_i !}+o_N(1)= q_{i,N}(h)+o_N(1)
\end{equation}for $0\leq h\leq L(N)$ (for the explanation of the $o_N(1)$ term, see the discussion after Proposition \ref{growth}). If, on the other hand, we have $i>s_0$ (namely, in the case where the function $x_i$ is sub-fractional), then \begin{equation}
    \max_{0\leq h\leq L(N)} |x_i(N+h)-x_i(N)|=o_N(1). 
\end{equation}In addition, we denote $p_{i,N}(h)=p_i(N+h)$ and similarly $\wt{p}_{i,N}(h)=\wt{p}_i(N+h)$ for every admissible value of $i$. Thus, we rewrite the expression in \eqref{Fu2} as \begin{equation}\label{asdfghjklequation}
    \underset{0\leq h\leq L(N)}{\E} F(w_N \prod_{i=1}^{s_0} u_i^{q_{i,N}(h)+p_{i,N}(h)} \prod_{i=s_0+1}^{s} u_i^{\wt{p}_{i,N}(h)}                       \G')
\end{equation}where we discarded the $o_N(1)$ terms, because $F$ is continuous. Here, $w_N=\prod_{i=s_0+1}^{s}u_i^{x_i(N)}$ but the explicit form of this term will not concern us, since we will only require that the element $w_N$ belongs to the underlying group $G'$ defining the nilmanifold $X'=\overline{u_1^{\R}\dots u_s^{\R}\G}$ 

In conclusion, we have reduced our problem to showing that given the nilmanifold $X'=\overline{u_1^{\R}\dots u_s^{\R}\G}$ (which is also equal to $\overline{u_1^{\Z}\dots u_s^{\Z}\G}$), the averages in \eqref{asdfghjklequation} converge. Here, the polynomials $q_{i, N}$ are defined in \eqref{Taylorexp} (they are essentially the Taylor polynomials of the Hardy field functions $x_i$), while the polynomials $p_{i, N},\wt{p}_{j, N}$ were defined by the relations $p_{i, N}=p_i(N+h) \text{ and }\ \wt{p}_{j, N}=\wt{p}_j(N+h)$, where the $p_i,\wt{p}_j$ are polynomials with real coefficients. We also recall that the polynomials $\wt{p}_i$ are such that every non-trivial linear combination of them is not an integer polynomial. Under all these assumptions,
we will show that the polynomial sequence (restricted to the range $0\leq h\leq L(N)$) inside the function $F$ is $\delta$-equidistributed for $N$ sufficiently large in the following step. We remark that the growth conditions \eqref{growthinequalitiessht} imposed on the function $L(t)$ will also play a crucial role in this.

{\bf Step 2: Using the quantitative equidistribution theorem:}
Let $Z\cong \T^{d}$ be the horizontal torus of the nilmanifold $X'=\overline{u_1^{\R}\dots u_s^{\R}\G}$ and let $\pi :X'\to Z$ denote the projection map.
Let $\delta>0$ be sufficiently small (in the sense that Theorem \ref{quantitative} is applicable). We assert that the finite polynomial sequence \begin{equation}\label{finitesequence}
    \Big(        \prod_{i=1}^{s_0} u_i^{q_{i,N}(h)+p_{i,N}(h)} \prod_{i=s_0+1}^{s} u_i^{\wt{p}_{i,N}(h)} \G'     \Big)_{0\leq h\leq L(N)}
\end{equation}is $\delta$-equidistributed on the nilmanifold $X'$ for $N$ sufficiently large. If the claim does not hold for a natural number $N$, then by Theorem \ref{quantitative}, there exists a real number\footnote{ The constant $M$ depends only on $\delta$, the nilmanifold $X'$ as well as the degrees of the polynomials $q_i,p_i$, which are all fixed in our arguments. The central property we need is that it is independent of the variable $N$.} $M>0$ and a non-trivial horizontal character $\chi_N$ of modulus $\leq M$ such that \begin{equation}\label{Fu7}
    \bignorm{ \chi_N\circ \pi ( \prod_{i=1}^{s_0} u_i^{q_{i,N}(h)+p_{i,N}(h)} \prod_{i=s_0+1}^{s} u_i^{\wt{p}_{i,N}(h)} \G')}_{C^{\infty}[L(N)]}\leq M. 
\end{equation} Thus, if our prior assertion fails, then the above relation would hold for infinitely many $N\in \N$.

Our first goal is to eliminate the dependence of the characters $\chi_N$ on the variable $N$.
Note that the function $\chi_N\circ \pi$ is a character on $\T^d$ of modulus $\leq M$ and, thus, has the form 
$$(t_1,...,t_{d})\to e(k_{1,N}t_1+\dots+k_{d,N} t_d)$$ for $k_{i,N}\in \Z $ with $|k_{1,N}|+\dots +|k_{d,N}|\leq M$. We also write $\pi(u_i)= (u_{i,1},...,u_{i,d})$ for the projections of the elements $u_i$ on the horizontal torus. Then, a straightforward computation allows us to rewrite \eqref{Fu7} as \begin{multline}\label{Fu3-}
    \bignorm{e\big(\sum_{i=1}^{s_0}(q_{i,N}(h)+p_{i,N}(h))(k_{1,N}u_{i,1}+\dots+ k_{d,N}u_{i,d})+\\
\sum_{i=s_0+1}^{s}(\wt{p}_{i,N}(h))(k_{1,N}u_{i,1}+\dots+ k_{d,N}u_{i,d})\big)}_{C^{\infty}[L(N)]}\leq M.
\end{multline}

Since there are only finitely many choices for the numbers $k_{1, N},...,k_{d, N}$, we have that, if our claim fails, there are $k_1,...,k_d\in \Z$, so that the inequality \begin{multline}\label{Fu3}
    \bignorm{e\big(\sum_{i=1}^{s_0}(q_{i,N}(h)+p_{i,N}(h))(k_{1}u_{i,1}+\dots+ k_{d}u_{i,d})+\\
\sum_{i=s_0+1}^{s}(\wt{p}_{i,N}(h))(k_{1}u_{i,1}+\dots+ k_{d}u_{i,d})\big)}_{C^{\infty}[L(N)]}\leq M.
\end{multline}holds for infinitely many $N\in \N$. We will also denote the horizontal character corresponding to the $d$-tuplet $(k_1,...,k_d)$ by $\chi$. Thus, we have eliminated the dependence of the character $\chi$ on $N$.

Denote $\wt{u_i}=k_1u_{i,1}+\dots+ k_du_{i,d}$. We will show that the above hypotheses imply that all the numbers $\wt{u_i}$ equal $0$.
Thus, suppose that this is not valid and we will reach a contradiction. We consider two cases:\\

\underline{Case 1:} Firstly, suppose that all of the numbers $\wt{u_i}$ with $1\leq i\leq s_0$ are zero,  which implies that the first summand in \eqref{Fu3} vanishes. Naturally, \eqref{Fu3} is simplified to \begin{equation}\label{Fu5}
    \bignorm{e\big(\sum_{i=s_0+1}^{s}\wt{p}_{i,N}(h)\wt{u_i}\big)}_{C^{\infty}[L(N)]}\leq M.
\end{equation}We recall here that we had defined $\wt{p}_{i,N}(h)=\wt{p}_i(N+h)$.
Let $Q(t)= \sum_{i=s_0+1}^{s}\wt{u_i}\wt{p}_{i}(t)$.
This is a linear combination of the polynomials $\wt{p}_i(t)$. However, this linear combination is not a polynomial in $\Q[t]$ due to our assumptions on the polynomials $\wt{p}_i(n)$, unless, of course, all the coefficients $\wt{u_i}$ (for $s_0+1\leq i\leq s$) in this combination are zero, which we have supposed to not be the case. Thus, $Q(t)$ has at least one irrational coefficient (except the constant term) and is equidistributed on $\T$. The relation \eqref{Fu5} implies that $\norm{e(Q(N+h))}_{C^{\infty}[L(N)]}\leq M$ for infinitely many $N$. It is not difficult to see by calculating the coefficients in $Q(N+h)$ that this fails for $N$ large enough.

\underline{Case 2:} Suppose now that at least one of the numbers $\wt{u_i}$ with $1\leq i\leq s_0$ is non-zero
Furthermore,
assume $l$ is a positive integer that is larger than the degrees of the polynomials $p_{i, N}(h),\wt{p}_{j, N}(h)$ (for all admissible values of the indices $i, j$) as well as the degrees of the functions $x_i$, but $l$ is also smaller than all the numbers $k_i$. Recall that we have picked $k_i$ to be very large in relation to the degrees of the polynomials $p_{i},\wt{p}_{j}$ and degrees of the functions $x_i$ in the beginning, thus we can find "many" such numbers $l$. The fact that $l$ is larger than the degrees of the functions $x_i$ combined with Proposition \ref{prop:basic} implies that $x_i^{(l)}(t)\to 0$, as $t\to+\infty$. 

For a number $l$ as above the coefficient of the term $h^{l}$ in the polynomial appearing in \eqref{Fu3} is equal to \begin{equation*}
    \frac{1}{l!}\sum_{i=1}^{s_0}x_i^{(l)}(N) \wt{u_i}
\end{equation*}and, thus, it does not depend on the polynomials $p_i,\wt{p}_j$. Using the definition of the smoothness norms, \eqref{Fu3} implies that \begin{equation*}
    L(N)^{l}\bignorm{\frac{1}{l!}\sum_{i=1}^{s_0}x_i^{(l)}(N) \wt{u_i}}_{\T}\leq M
\end{equation*}for infinitely many $N\in \N$. The last inequality becomes \begin{equation*}
     L(N)^{l}\big|\sum_{i=1}^{s_0}x_i^{(l)}(N)\wt{u_i}\big|\leq l!M,
\end{equation*}for large enough $N$, because all functions $x_i^{l}(t)$ go to 0. However, the Hardy field function inside the absolute value above has the same growth rate as the function $x_1^{(l)}(t)$, since the functions $x_1,...,x_{s_0}$ are strongly non-polynomial and have distinct growth rates (recall that $x_1$ has the largest growth rate among the $x_i$), unless, of course, $\wt{u}_1=0$. If the latter does not hold, we get \begin{equation*}
     \big|x_1^{(l)}(N)\wt{u}_1  \big|\leq \frac{C}{L(N)^l}
\end{equation*} for infinitely many $N$ and some constant $C$, which contradicts \eqref{growthinequalitiessht}. Thus, we eventually deduce that $\wt{u}_1=k_1u_{1,1}+\dots+ k_du_{1,d}=0$. Repeating the same argument, we get inductively that $\wt{u}_i=k_1u_{i,1}+\dots + k_du_{i,d}$=0 for all $1\leq i\leq s_0$, which is a contradiction.

To summarize, we have shown that if the sequence in \eqref{finitesequence} is not $\delta$-equidistributed for all large enough $N$, then all the numbers $\wt{u}_i=k_1u_{i,1}+\dots+ k_du_{i,d}$ are zero. Equivalently, we have $\chi\circ \pi (u_i)=0$ for all $1\leq i\leq s$. This implies that the character $\chi$ is the trivial character on $X'$. Indeed, the character $\chi$ annihilates all elements $u_1^{n_1}\cdots  u_s^{n_s}\G$, where $n_1,...,n_s\in \Z$ and by density of those elements on $X'$ (recall our assumption that $X'$ is also equal to the nilmanifold $\overline{u_1^{\Z}...u_s^{\Z}\G}$), $\chi$ is zero everywhere. This is a contradiction (the horizontal characters appearing when we applied Theorem \ref{quantitative} are assumed to be non-trivial).

In conclusion, we have that the finite polynomial sequence in \eqref{finitesequence} is $\delta$-equidistributed for $N$ sufficiently large. Thus, we conclude that the averages in \eqref{Fu2} are $\delta\norm{F(w_N \cdot  )}_{\text{Lip}(X')}=\delta\norm{F}_{\text{Lip}(X')}$ close to the quantity $
    \int_{X'} F(w_N x)\ dm_{X'}(x)$. The action of $w_N$ on $X'$ preserves the Haar measure of $X'$, so we get that the last integral is equal to $\int_{X'} F(x) \ dm_{X'}(x)$. Taking $\delta\to 0$, we finish the proof.
    \end{proof}

\begin{proof} [Proof of Theorem \ref{equidistribution}]
As we explained in the previous section (before the statement of Lemma \ref{removerounding}), the first part follows from the second part (see also \cite[Lemma 5.1]{Fraeq}) and, in turn,
this second part follows using similar arguments as in the proof of Theorem \ref{pointwise}. We only highlight the main differences here. All the disparities appear in the part where we reduce Proposition \ref{impprop} to Lemma \ref{Lemma before step 2}.\\
a) In \eqref{Fu1}, all the functions $f_1,...,f_s$ satisfy \eqref{P1} (there are no functions among the $f_i$ that satisfy $\lim\limits_{t\to+\infty}|f_i(t)|<\infty$). We also have $k=s$.\\
b) We do not have to make the reduction to the case where the polynomials $\wt{p}_1,...,\wt{p}_{\ell}$ are linearly independent. There cannot be a non-trivial linear combination of them that is zero, because that would violate \eqref{P1}. \\
c) The limit of the averages is again $\int_{X'} F(x) \ dm_{X'}(x)$, where $X'=\overline{u_1^{\R}... u_s^{\R}\G}$ by Lemma \ref{Lemma before step 2}. We would like to show that the limit is equal to $\int_{X''} Fdm_X''$, where $X''$ is the nilmanifold $\overline{b_1^{\R}... b_k^{\R}\G}$. Recall that each $u_i$ is equal to $b_1^{c_{i,1}}... b_k^{c_{i,k}}$ (by \eqref{linearexpansion}) and the numbers $c_{i,j}$ form an invertible $k\times k$ matrix (due to the linear independence assumption on the original functions $a_1,...,a_k$). Thus, we can also write $b_i=\prod_{j=1}^{k}u_i^{c'_{i,j}}$ for some numbers $c'_{i,j}$ (here, we also use that the elements $b_i$ are pairwise commuting). Combining the above, we have that $b_1^{\R}... b_k^{\R}=u_1^{\R}... u_k^{\R}$ and thus the closures of their projections on $G/\G$ define the same subnilmanifold.
\end{proof}

\subsection{Proof of Theorem \ref{simple}}
Finally, we provide a proof of Theorem \ref{simple}. We use the following proposition from \cite{tsinas}. Although it will not be used in the proof, we have to assume below that the Hardy field $\mathcal{H}$ that we work with is closed under composition and compositional inversion of functions, since the following proposition was proven under this assumption.

\begin{proposition}\cite[Proposition 3.1]{tsinas}\label{factors}
Let $\mathcal{H}$ be a Hardy field that contains the field $\mathcal{LE}$ of logarithmico-exponential functions and is closed under composition and compositional inversion of functions (when defined).
Assume that the functions $a_1,...,a_k\in\mathcal{H}$ have polynomial growth and suppose that the following two conditions hold:

i) The functions $a_1,...,a_k$ dominate the logarithmic function $\log t$.

ii) The pairwise differences $a_i-a_j$ dominate the logarithmic function $\log t$ for any $i\neq j$.

Then, there exists a positive integer $s$, such that for any measure preserving system $(X,\m,T)$, functions $f_1\in L^{\infty}(\m)$ and $f_{2,N},...,f_{k,N}\in L^{\infty}(\m)$, all bounded by $1$, with $f_1\perp Z_{s}(X)$, the expression \begin{equation}\label{factor}
   \sup_{|c_{n}|\leq 1} \norm{\underset{1\leq n\leq N}{\E}\ c_{n}\  T^{\floor{a_1(n)}} f_1 \cdot T^{\floor{a_2(n)}}f_{2,N}\dots  T^{\floor{a_k(n)}}f_{k,N}}_{L^2(\m)}
\end{equation}converges to 0, as $N\to+\infty$.
\end{proposition}

\begin{proof}[Proof of Theorem \ref{simple}]

Using a standard ergodic decomposition argument, we may assume that the system $(X,\m, T)$ is ergodic. We can also rescale the functions $f_i\in L^{\infty}(\m)$ so that they are 1-bounded. 
Our first objective is to apply Proposition \ref{factors}, in order to reduce the problem to the case where the system $X$ is a nilsystem. If the functions $a_1,...,a_k$ are such that the conditions of Proposition \ref{factors} are satisfied, then this can be done instantly. If this does not hold, we have to perform a series of reductions to be able to apply Proposition \ref{factors}. We do this in 2 steps:

a) Firstly, assume there exists one function among the $a_i$ (say $a_1$ for simplicity), which has growth rate smaller than or equal to $\log t$. Then, using \eqref{P1} and \eqref{P2}, we deduce that $a_1$ converges monotonically to some real number $c$ and the integer part of $a_1(n)$ becomes a constant. Thus, the asymptotic behavior of the averages in \eqref{ergodicaverages} is the same, if we substitute the term $T^{\floor{a_1(n)}f_1}$ with the term $T^{\floor{c}}f_1$. Consequently, we only need to show that the averages \begin{equation*}
    \underset{1\leq n\leq N}{\E} T^{\floor{a_2(n)}}f_2 \dots T^{\floor{a_k(n)}}f_k
\end{equation*}converge in norm. Repeating the same argument, we eliminate all functions $a_i$ that grow slower than $\log t$.

b) Due to the reduction in the previous step, we have a sub-collection of the original functions, so that all functions in this new set dominate $\log t$. We will denote this collection by $a_1,...,a_k$ again, and our task is to show that the averages \begin{equation*}
    \underset{1\leq n\leq N}{\E} T^{\floor{a_1(n)}}f_1 \dots T^{\floor{a_k(n)}}f_k
\end{equation*} converge in mean (for all systems). Our next objective is to eliminate pairs of functions, whose difference grows slower than $\log t$ so that we can ultimately apply Proposition \ref{factors}.

Assume that two of the functions (say $a_1,a_2$) are such that their difference is dominated by $\log t$. We observe that the function $a_1(t)$ goes to $\pm \infty$ as $t\to+\infty$, since it dominates $\log t$. In that case, the function $a_1(t)$ satisfies \eqref{P1} and by Theorem \ref{Bosh}, the sequence $a_1(n)$ is equidistributed $(mod\  1)$. Observe that since $a_1-a_2$ must satisfy \eqref{P2}, we must have $a_2(t)=a_1(t)+c+x(t)$, where the function $x(t)\in \mathcal{H}$ converges to $0$ monotonically and $c$ is a real number. Thus, for $t\in \R$ sufficiently large we have \begin{equation}\label{frbl}
    \floor{a_2(t)}=\floor{a_1(t)+x(t)+c}=\floor{a_1(t)} +\floor{c}+\e(t)
\end{equation}where $\e(t)\in \{0,\pm 1,\pm 2\}$ and the value of $\e(t)$ depends on whether the inequalities 
$$\{a_1(t)+c\}+\{x(t)\}\leq 1$$ and $$\{a_1(t)\}+\{c\}\leq 1$$ hold or not, as well as whether $x(t)$ is eventually positive or negative.

Define $A_z=\{t\in \R, \e(t)=z\}$ for $z\in\{0,\pm 1,\pm 2\}$. Then, we see that our multiple averages are equal to the sum \begin{equation*}
    \sum_{z\in \{0,\pm 1,\pm 2\}} \underset{1\leq n\leq N}{\E} {\bf 1}_{A_z}(n)\  T^{\floor{a_1(n)}}(f_1\cdot T^{\floor{c}+z}f_2)\cdot T^{\floor{a_3(n)}}f_3\dots T^{\floor{a_k(n)}} f_k.
\end{equation*}For a fixed $z$, we want to show that the corresponding average converges. For $n\in \N$ large enough, we will approximate the sequence ${\bf 1}_{A_z}(n)$ by sequences of the form $F(\{a_1(n)\})$, where $F$ is a continuous function. 

We establish this for $z=0$ (the other cases follow similarly).  Assume that $x(t)$ decreases to 0 (the other case) is similar, which means that $x(t)$ is eventually positive and also $\{x(t)\}=x(t)$ for $t$ sufficiently large. In addition, we can also assume that $c$ is positive.
Observe that for $t\in A_{0}$, we have $$\floor{a_2(t)}=\floor{a_1(t)}  +\floor{c}$$ by the definition of $A_0$. This is equivalent to the inequalities\begin{align*}
    &\{a_1(t)+c\}+\{x(t)\}\leq 1\\
    &\{a_1(t)\}+\{c\}\leq 1
\end{align*}which can be condensed into
\begin{equation}
    \{a_1(t)\}+\{x(t)\}\leq 1-\{c\},
\end{equation}since we assumed for simplicity that $x(t)$ is eventually positive. To summarize, we have shown that \begin{equation}\label{main inequality for approximation}
    n\in A_0\iff \{a_1(n)\}+\{x(n)\}\leq 1-\{c\}
\end{equation}

Let $\e>0$ be a small number. Since we have that the function $x(t)$ decreases to $ 0$, we have that $\{x(t)\}<\e$ for $t$ large enough. Consider the set \begin{equation*}
    A_{\e}=\{n\in \N\colon \{a_1(n)\}\leq 1-\{c\}-\e \}.
\end{equation*}Then, for sufficiently large values of $n$, we observe that if $n\in A_{\e}$, then the inequality\begin{equation*}
    \{a_1(t)\}+\{x(t)\}\leq 1-\{c\}
\end{equation*}holds as well. Namely, $A_\e \subseteq A_0$.
Let us denote $B_{\e}=[0,1-{c}-\e]$ for convenience and observe that \begin{equation*}
    \bol_{A_{\e}}(n)=\bol_{B_\e}(\{a_1(n)\}).
\end{equation*}Now we approximate the function $\bol_{B_\e}$ by a continuous function in the uniform norm, where $\bol_{B_{\e}}$ is considered a function on the torus $\T$ in the natural way. We can define a continuous function  on $\T$, such that $F_{\e}$ agrees  with $\bol_{B_\e}$ on the set\footnote{In the case that $c$ is an integer, we make natural modifications to this set. For example, one could define the function $F_{\e}$ so that it agrees with $\bol_{\B_{\e}}$ on $[\e,1-2\e]$. Basically, we only require the function $F_{\e}$ to agree with $\bol_{B_{\e}}$ on a set of measure $1-O(\e)$ for our argument to work.} \begin{equation*}
    [\e,1-\{c\}-2\e ]\cup [1-\{c\},1-\e]
\end{equation*}and such that $\norm{F_{\e}-\bol_{B_\e} }_{\infty}\leq 2 $.
We suppose that $\e$ is small enough so that these intervals are well-defined. Observe that $\bol_{B_{\e}}$ is equal to 1 on the first interval of this union and equal to 0 on the second interval. 

Observe that \begin{equation*}
    A_0\setminus A_{\e}=\{n\in \N\colon  1-\{c\}-\e < \{a_1(n)\}  \leq 1-\{c\}-\{x(n)\} \}\subseteq \{n\in \N\colon \{a_1(n)\}\in [1-\{c\}-\e,1-\{c\} ] \}.
\end{equation*}
Since the function $a_1(t)$ is equidistributed modulo 1, we conclude that the set $A_0\setminus A_{\e}$ has upper density at most $\e$.
Therefore, we have \begin{multline}\label{1st inequality}
     \bignorm{\underset{1\leq n\leq N}{\E} {\bf 1}_{A_0}(n)\  T^{\floor{a_1(n)}}(f_1\cdot T^{\floor{c}}f_2)\cdot T^{\floor{a_3(n)}}f_3\dots T^{\floor{a_k(n)}} f_k -\\ \underset{1\leq n\leq N}{\E} \bol_{B_{\e}}(\{a_1(n)\})\  T^{\floor{a_1(n)}}(f_1\cdot T^{\floor{c}}f_2)\cdot T^{\floor{a_3(n)}}f_3\dots T^{\floor{a_k(n)}} f_k}_{L^2(\m)}\leq \\
     \E_{1\leq n\leq N} \bol_{A_0\setminus A_{\e}}(n)\norm{   \bol_{A_{\e}}-1_{A_{\e}\setminus A_0} }_{\infty} \leq 2\e+o_N(1),
\end{multline}where we used the fact that $ \bol_{A_{\e}}(n)=\bol_{B_\e}(\{a_1(n)\})$ for all $n\in \N$, the trivial bound for the values of $n\in A_{0}\setminus A_{\e}$ and the fact that the set $A_{0}\setminus A_{\e}$ has upper density at most $\e$.

We do a similar comparison for the averages weighted by $F_{\e}(\{a_1(n)\})$ and $1_{B_{\e}}(\{a_1(n)\})$. To be more specific, we reiterate that the functions $1_{B_{\e}}$ and $F_{\e}$ agree on the set  \begin{equation*}
    [\e,1-\{c\}-2\e ]\cup [1-\{c\},1-\e].
\end{equation*} Accordingly, we have $\bol_{B_{\e}}(\{a_1(n)\})=F_{\e}(\{a_1(n)\})$, unless \begin{equation*}
    \{a_1(n)\}\in [0,\e)\cup (1-\{c\}-2\e,1-\{c\}     )\cup (1-\e,1).
\end{equation*}Let $C_{\e}$ denote the set of $n\in \N$ for which $\{a_1(n)\}$ belongs to this union.
This union has measure $4\e$, which implies that the upper density of  $C_{\e}$ is at most $4\e$ (since $a_1(n)$ is equidistributed modulo 1). Hence, we infer that \begin{multline}\label{2nd inequality}
     \bignorm{\underset{1\leq n\leq N}{\E} F_{\e}(\{a_1(n)\})\  T^{\floor{a_1(n)}}(f_1\cdot T^{\floor{c}}f_2)\cdot T^{\floor{a_3(n)}}f_3\dots T^{\floor{a_k(n)}} f_k -\\ \underset{1\leq n\leq N}{\E} \bol_{B_{\e}}(\{a_1(n)\})\  T^{\floor{a_1(n)}}(f_1\cdot T^{\floor{c}}f_2)\cdot T^{\floor{a_3(n)}}f_3\dots T^{\floor{a_k(n)}} f_k}_{L^2(\m)}\leq \\
     \E_{1\leq n\leq N}     1_{C_{\e}}(n)\norm{F_{\e}-1_{B_{\e}}}_{\infty} \leq 8\e +o_N(1),
\end{multline}where we utilized the fact that $\bol_{B_{\e}}(\{a_1(n)\})=F_{\e}(\{a_1(n)\})$ for all $n$ on the complement of $C_{\e}$, the trivial bound for the values of $n\in C_{\e}$ and the fact that $C_{\e}$ has upper density at most $4\e$. 

Combining \eqref{1st inequality} and \eqref{2nd inequality}, we deduce that 
\begin{multline}\label{3rd inequality}
     \bignorm{\underset{1\leq n\leq N}{\E} {\bf 1}_{A_0}(n)\  T^{\floor{a_1(n)}}(f_1\cdot T^{\floor{c}}f_2)\cdot T^{\floor{a_3(n)}}f_3\dots T^{\floor{a_k(n)}} f_k -\\ \underset{1\leq n\leq N}{\E} F_{\e}(\{a_1(n)\})\  T^{\floor{a_1(n)}}(f_1\cdot T^{\floor{c}}f_2)\cdot T^{\floor{a_3(n)}}f_3\dots T^{\floor{a_k(n)}} f_k}_{L^2(\m)}\leq 10\e+o_N(1).
\end{multline}

Taking $\e\to 0$, we deduce that it is sufficient to verify that the averages \begin{equation*}
    \underset{1\leq n\leq N}{\E} F(\{a_1(n)\})\ T^{\floor{a_1(n)}}(f_1 \cdot T^{\floor{c}}f_2)\dots T^{\floor{a_k(n)}} f_k
\end{equation*}converge for any continuous function $F$ on $\T$. This would imply that the averages 
\begin{equation*}
    \underset{1\leq n\leq N}{\E} 1_{A_0}(n)\ T^{\floor{a_1(n)}}(f_1 \cdot T^{\floor{c}}f_2)\dots T^{\floor{a_k(n)}} f_k
\end{equation*}converge in norm.

After approximating $F$ by trigonometric polynomials (in the uniform norm), it suffices to show that the averages
\begin{equation*}
   \underset{1\leq n\leq N}{\E} e(l_1a_1(n))\ T^{\floor{a_1(n)}}(f_1\cdot T^{\floor{c}}f_2)\dots T^{\floor{a_k(n)}} f_k
\end{equation*} converge in norm for any $l_1\in \Z$. Note that the function $a_2(t)$ has vanished and its role has been replaced by the sequence $e(l_1a_1(n))$.

We repeat this process until we eliminate all pairs of functions, whose difference grows slower than $\log t$, where at each step our averages are multiplied by a sequence of the form $e(l_ia_i(n))$ ($l_i\in \Z$).
After finitely many iterations, our problem eventually reduces to the following: let $a_1,...,a_k$ satisfy \eqref{P1} or \eqref{P2} and let $b_1,...,b_m$ be a subset of $\{a_1,...,a_k\}$, so that the functions $b_1,...,b_m$ satisfy the hypotheses of Proposition \ref{factors}. Then, for any integers $l_1,...l_k$, the averages \begin{equation*}
    \underset{1\leq n\leq N}{\E} e(l_1a_1(n)+\dots+l_ka_k(n)) \ T^{\floor{b_1(n)}}f_1\dots T^{\floor{b_m(n)}}f_m
\end{equation*}converge in $L^2(\m)$ for all functions $f_1,\dots, f_m\in L^{\infty}(\m)$.

Now we can apply Proposition \ref{factors} and use a standard telescopic argument to show that the limiting behavior of the above averages does not change if we replace the functions $f_i$ by their projections to the factor $Z_s(X)$ (the number $s$ is the one given by Proposition \ref{factors}). However, by Theorem \ref{structure}, the factors $Z_s(X)$ are inverse limits of $s$-step nilsystems. Thus, by another standard limiting argument, we may reduce to the case that the space $X$ is a nilmanifold and $\m$ is its Haar measure, while the transformation $T$ is the action (by left multiplication) of an element $g$ on $X$. Finally, we can approximate the functions $f_i$ by continuous functions and reduce our problem to the following:

If $X=G/\G$ is a nilmanifold with $g\in G$ and the functions $a_1,...,a_k,b_1,...,b_m\in \mathcal{H}$ are as above, then for any continuous functions $f_1,...,f_m$ the averages 
\begin{equation*}
    \underset{1\leq n\leq N}{\E} e(l_1a_1(n)+\dots+l_ka_k(n))\ f_1(g^{\floor{b_1(n)}}x)\dots f_k(g^{\floor{b_m(n)}}x)
\end{equation*}converge in mean. 

We show that these averages converge pointwise for every $x\in X$. We recall that the functions $b_1,...,b_m$ belong to the set $\{a_1,...,a_k\}$ (this is the only thing that we will need to use for the rest of the proof).

First of all, it suffices to show that the averages \begin{equation*}
    \underset{1\leq n\leq N}{\E} e(l_1a_1(n)+\dots+l_ka_k(n)) \ f_1(g^{{b_1(n)}}x)\dots f_k(g^{{b_m(n)}}x)
\end{equation*}converge pointwise, where $X=G/\G$ is such that $G$ is connected, simply connected nilpotent Lie group (basically, we can remove the integer parts appearing in the iterates). This follows by standard modifications in the proof of Lemma \ref{removerounding}  (the fact that we have the coefficients $e(l_1a_1(n)+\dots+l_ka_k(n))$ in the final expression does not affect the argument), so we omit the details.

Now, observe that we can write the above averages in the form \begin{equation*}
    \underset{1\leq n\leq N}{\E} F_0(g_0^{l_1a_1(n)+\dots+l_ka_k(n)}\tilde{x})\ F_1(\tilde{g}^{{b_1(n)}}\tilde{x})\dots F_k(\tilde{g}^{{b_m(n)}}\tilde{x}),
\end{equation*}where $g_0=(1_{\T},e_G)$ and $\tilde{g}=(1_{\T},g)$ act on the product nilmanifold $\T\times X$, the point $\tilde{x}$ is just $(\Z,x)$ and the functions $F_i$ are defined by $$F_0(y\Z,a\G)=e(y) \ \text{ and } F_i(y\Z,a\G)=f_i(a\G)\ \text{for } i\geq 1.$$These are continuous functions on $\T\times X$. The functions $l_1a_1(t)+\dots+l_ka_k(t),b_1(t),...,b_m(t)$ satisfy the hypotheses of Theorem \ref{pointwise} (since the functions $a_1,...,a_k$ do) and the result follows.
\end{proof}

\begin{appendix}\label{appendix}

    \section{Hardy field functions in short intervals}\label{hardyappendix}
\subsection{Growth rates of Hardy field functions}

 All of the results presented below were proven in \cite{tsinas}, and thus we omit their proofs. We refer the reader to the example in Section \ref{preparations proof}, where we establish Theorem \ref{pointwise} in the case where we have two simple functions. In that example, we do not need any special lemmas to show that we can find a common Taylor expansion, because we can perform the calculations by hand. However, in the proofs of Theorems \ref{equidistribution} and \ref{pointwise}, we need to show that we can always do the same common polynomial expansion for general functions.

The first two propositions are some elementary facts concerning the growth rates of derivatives of functions in a Hardy field.

\begin{proposition}\label{prop:basic}\cite[Proposition A.1]{tsinas}
Let $f\in\mathcal{H}$ have polynomial growth. Then, for any natural number $k$, we have \begin{equation*}
    f^{(k)}(t) \ll \frac{f(t)}{t^k}.
\end{equation*}
In addition, if $ t^{\delta}\prec f(t)$ for some $\delta>0$, we have \begin{equation*}
    f'(t) \sim \frac{f(t)}{t}.
\end{equation*}
\end{proposition}

The above proposition establishes that if we have a function in $\mathcal{H}$ that has polynomial growth, then its derivatives of large enough order will be functions that converge to 0. The next lemma implies that a particular growth relation holds between consecutive derivatives (of large enough order).

\begin{proposition}\label{growth}\cite[Proposition A.2]{tsinas}
Let $f\in\mathcal{H}$ be strongly non-polynomial with $f(t)\succ \log t$. Then, for $k$ sufficiently large, we have \begin{equation*}
     1\prec |f^{(k)}(t)|^{-1/k}\prec  |f^{(k+1)}(t)|^{-1/(k+1)}\prec t.
\end{equation*}
\end{proposition}

Let us demonstrate how this proposition is used to get a polynomial expansion in short intervals for a single function.
Let $a\in\mathcal{H}$ be a strongly non-polynomial function that satisfies the growth condition $a(t)\succ \log t$. Let $k$ be a positive integer, that is large enough so that we can apply the two preceding propositions. We argue that we can find a function $L(t)$ (not necessarily in $\mathcal{H}$) such that \begin{equation}\label{L(t)}
    |a^{(k)}(t)|^{-1/k}\prec L(t)\prec  |a^{(k+1)}(t)|^{-1/(k+1)}.
\end{equation}For instance, the geometric mean of the functions $|a^{(k)}(t)|^{-1/k}$ and $|a^{(k+1)}(t)|^{-1/(k+1)}$ is a suitable choice for our purposes.

We will examine the function $a$ in intervals of the form $[N, N+L(N)]$ and approximate it by a polynomial, which will vary with $N$. Observe that if $0\leq h\leq L(N)$, then we have \begin{equation*}
    a(N+h)=a(N)+\dots+\frac{h^ka^{(k)}(N)}{k!} +\frac{h^{k+1}a^{(k+1)}(\xi_{h,N})}{(k+1)!}
\end{equation*}for some $\xi_{h,N} \in [N,N+h]$. Using the largeness of $k$, Proposition \ref{growth} implies that $|a^{(k+1)}(t)|\to 0$ monotonically (the monotonicity follows from the fact that the function $a^{(k+1)}(t)$ belongs to $\mathcal{H}$). Then, for $N$ sufficiently large, \begin{equation*}
    \Big|\frac{h^{k+1}a^{(k+1)}(\xi_{h,N})}{(k+1)!}\Big|\leq\Big| \frac{L(N)^{k+1}a^{(k+1)}(N)}{(k+1)!}\Big|\prec 1,
\end{equation*}because of \eqref{L(t)}. Furthermore, we have that \begin{equation*}
    \Big|\frac{L(N)^{k}a^{(k)}(N+L(N))}{k!}\Big|\to +\infty.
\end{equation*}Indeed, since $L(t)$ is a sub-linear function by Proposition \ref{growth}, we infer that the two functions $a^{(k)}(t+L(t))$ and $a^{(k)}(t)$ have the same growth rate and thus we only need to prove that \begin{equation}\label{epr}
   \Big|\frac{L(N)^{k}a^{(k)}(N)}{k!}\Big|\to +\infty.  
\end{equation}This follows similarly as above. To summarize, we have \begin{equation}\label{crpsht}
     a(N+h)=a(N)+\dots+\frac{h^ka^{(k)}(N)}{k!}+o_N(1),\ \text{ for }\ \ 0\leq h\leq L(N).
\end{equation}

Therefore, functions that satisfy \eqref{L(t)} have the following distinctive property: the sequence $a(n)$, when restricted to the intervals $[N, N+L(N)]$ as above, is asymptotically equal to a polynomial sequence (that depends on $N$) of degree exactly $k$. This motivates us to study the properties of functions that satisfy \eqref{L(t)}. The main goal is to accomplish the same for several functions $a_1,..., a_{m}$ in a Hardy field $\mathcal{H}$. This is relatively straightforward to do by hand in explicit examples, like the one in Section \ref{preparations proof}. In the more abstract setting, if we manage to show that we can find a function $L(t)$, so that \eqref{L(t)} is satisfied for all functions $a_1,..., a_{m}$ (the integer $k$ is allowed to be different for each function), then we will establish that a polynomial expansion like the one in \eqref{crpsht} holds for all the functions $a_1,..., a_{m}$ simultaneously. We will introduce some notions shortly that will assist us in this endeavor.

\subsection{The sub-classes $S(a,k)$}
 Let $a\in\mathcal{H}$ be a strongly non-polynomial function such that $a(t)\gg t^{\delta}$, for some $\delta>0$ (namely, we exclude sub-fractional functions).
 For $k\in \N$ sufficiently large (we only require that $a^{(k)}(t)\to 0$), we define the subclass $S(a,k)$ of $\mathcal{H}$ as
\begin{equation*}
    S(a,k)=\{g\ \colon g(t)\prec t \ \text{ and } \ \  |a^{(k)}(t)|^{-\frac{1}{k}}\preceq g(t)\prec |a^{(k+1)}(t)|^{-\frac{1}{k+1}} \},
\end{equation*}
where the notation $g(t)\preceq f(t)$ signifies that the limit $\lim\limits_{t\to\infty} |f(t)/g(t)|$ is non-zero. The purpose of the classes $S(a,k)$ is to characterize the growth relation \eqref{L(t)}.
We will use the following lemma.

\begin{lemma}\label{basic}\cite[Lemma A.3]{tsinas}
Let $a\in\mathcal{H}$ be a strongly non-polynomial function with $a(t)\gg t^{\delta}$, for some $\delta>0$.\\
i) The class $S(a,k)$  is non-empty, for $k$ sufficiently large. \\
ii) For any $0< c< 1$ sufficiently close to 1, there exists $k_0\in \N$ (depending on $c$), such that the function $t\to t^c$ belongs to $S(a,k_0)$.  \\
iii) The class  $S(a,k)$ does not contain all functions of the form $t\to t^c$, for $c$ sufficiently close to 1.
\end{lemma}

A naive way to think of the sub-classes is like a sequence of disjoint intervals on a line (with no gaps between consecutive intervals). Property ii) in the above lemma implies that each function of the form $t^c$ for $c$ close to 1 belongs to a unique $S(a,k)$. We can demonstrate that this actually holds if the fractional power $t^c$ is replaced by any function $g$ satisfying a growth condition of the form $t^{c_1}\prec g(t)\prec t^{c_2}$, where $c_1$ must be sufficiently close to 1.

\begin{proposition}\label{common Taylor expansion}
    Let $a_1,\dots, a_k$ be strongly non-polynomial functions in $\mathcal{H}$ of polynomial growth, such that all the functions $a_i$ dominate some fractional power $t^{\delta}$ for some $\delta>0$. There exists $0<C<1$ depending only on the functions $a_1,\dots,a_k$, such that if the function $L(t)$ satisfies \begin{equation*}
        t^C\prec L(t)\prec t^{1-\e}
    \end{equation*}for some $\e>0$, then there exist positive integers $k_i$ (that depend on $L(t)$), such that $L(t)\in S(a_i,k_i)$ for every $i\in\{1,\dots,k\}$. In addition, for any  positive real number $M$, there exists a constant $A=A(M,a_1,\dots,a_k)\in (0,1)$, such that if 
     \begin{equation*}
        t^A\prec L(t)\prec t^{1-\e}
    \end{equation*}for some $\e>0$, then we have $k_i>M$ for every $i\in \{1,\dots, k\}$.
\end{proposition}

\begin{proof}
It is apparent that we only need to establish the assertion in the case $k=1$ (namely, when we have only one function).
   Therefore, we fix a strongly non-polynomial function $a$ that is not sub-fractional and recall that by Lemma \ref{basic}, there exists a constant $C<1$ depending only on $a(t)$, such that every function of the form $t^c$ with $c>C$ belongs to the class $S(a, n_c)$ for some natural number $n_c$. Now, assume that the function $L(t)$ satisfies \begin{equation}\label{c_1c_2}
    t^{C}\prec L(t)\prec t^{c_1} 
\end{equation}for some $C<c_1<1$. Then, because both $t^{C}$ and $t^{c_1}$ belong to the sub-classes $S(a,n_{C})$ and $S(a,n_{c_1})$ respectively for some $n_C,n_{c_1}\in \N$, we get that $L(t)$ belongs to $S(a,n_3)$ for some integer $n_3$ that satisfies $n_C\leq n_3\leq n_{c_1}$.

Now we establish the second part. Let $M$ be a fixed real number and consider a fractional power $t^{c_2}$ with $C<c_2<1$, so that $t^{c_2}$ belongs to $S(a,n_{c_2})$ for some $c_2>M$. Such a fractional power exists, which is evident by combining the second and third statements of Lemma \ref{basic}. Thus, if $L(t)$ satisfies\begin{equation*}
        t^{{c_2}}\prec L(t)\prec t^{1-\e}
    \end{equation*}for some $\e>0$, we have that $L(t)\in S(a,k')$ (by the first part) for a positive integer $k'$ with $k'\geq n_{c_2}>M$. The claim follows.
\end{proof}

The first part of Proposition \ref{common Taylor expansion} implies that if we are given functions $a_1,\dots,a_k$ that satisfy the hypotheses, then we can find a sub-linear function $L(t)$, such that $L(t)\in S(a_i,k_i)$. This asserts that the function $a_i$ will be approximated by a polynomial of degree $k_i$ in short intervals of the form $[N, N+L(N)]$, for every $i\in \{1,\dots,k\}$. Furthermore, the second part establishes that we can make the degrees $k_i$ of the Taylor polynomials arbitrarily large, as long as we take the function $L(t)$ to grow "sufficiently fast" (faster than some appropriate power $t^C$ with $C<1$).

The sub-classes $S(a,k)$ were defined for functions that are not sub-fractional. The above argument does not extend to these
 latter functions.
As an example, let us fix a number $\delta$ with $0<\delta<1$ and a sub-fractional function $a\in\mathcal{H}$. If we consider the function $L(t)=t^{\delta}$ and try to repeat the same approximations to obtain an analog of \eqref{crpsht}, we run into an issue. Clearly, it is easy to see that \begin{equation*}
    \max_{0\leq h\leq L(N)}|a(N+h)-a(N)|=o_N(1),
\end{equation*}using the mean value theorem. Thus, the sequence $a(n)$, when it is restricted to the interval $[N, N+L(N)]$, is $o_N(1)$ close to the value $a(N)$, which signifies that it is approximately equal to a constant on this interval (or equivalently, all polynomial expansions we get are of degree 0). This could be circumvented if we considered sub-linear functions $L(t)$ that grow faster than all the powers $t^{\delta}, 0<\delta <1$, such as the function $t/ \log t$. If we do this however, the growth condition \eqref{L(t)} can never hold for functions that are not sub-fractional\footnote{Concerning the problem of finding characteristic factors for ergodic averages involving Hardy field iterates, there was a workaround for this issue in \cite{tsinas} using a double-averaging trick. Unfortunately, the same argument breaks down in the setting of pointwise convergence on nilmanifolds. See also Remark \ref{diophantine}.} (in simple terms, there can be no polynomial approximation of finite degree). We omit the specific details of this deduction.

\subsection{Decomposing  Hardy field functions}

We consider a Hardy field $\mathcal{H}$ that contains the polynomials and
let $a$ be a function in $\mathcal{H}$. We partition $\mathcal{H}$ into equivalence classes by the relation $f\sim g$, which is equivalent to saying that the limit of $f(t)/g(t)$ as $t\to +\infty$ is a non-zero real number. In simple terms, $f,g$ are in the same equivalence class if and only if they have the same growth rate. We put the zero function in its own equivalence class.

We will define the strongly non-polynomial growth rate of a function $a\in\mathcal{H}$ as follows:\\
i) If $a$ is a strongly non-polynomial function (recall the definition in Section \ref{background}), we define it to be the equivalence class of $a$.\\
ii) If $a$ is not strongly non-polynomial, then it can be written in the form $p(t)+x(t)$, where $p(t)$ is a polynomial and $x(t)$ is a strongly non-polynomial function (or the zero function) with $x(t)\prec  p(t)$. Observe that $x(t)$ is a function in $\mathcal{H},$ since our Hardy field contains the polynomials. We define the strongly non-polynomial growth rate of $a$ as the equivalence class of the function $x\in\mathcal{H}$.

The strongly non-polynomial growth rate is defined for any function $a\in\mathcal{H}$. It is well defined, in the following sense: consider a function $a\in\mathcal{H}$ like in case ii) above, which has two different representations as $p_1(t)+x_1(t)$ and $p_2(t)+x_2(t)$, where $p_1,p_2$ are polynomials, $x_1,x_2$ are strongly non-polynomial and $x_1(t)\prec p_1(t)$ and $x_2(t)\prec p_2(t)$. Then, we must have $x_1(t)\sim x_2(t)$. An example where such distinct representations may exist is the function $a(t)=t^2+t+t^{3/2}$. We can choose $p_1(t)=t^2,x_1(t)=t^{3/2}+t$ and $p_2(t)=t^2+t,x_2(t)=t^{3/2}$. While $x_1\neq x_2$, these two functions have the same growth rate.

A simple observation is that, if a function $a\in\mathcal{H}$ is written in the form $p(t)+x(t)$, where $p$ is polynomial and $x$ is strongly non-polynomial, then the functions $a$ and $x$ have the same strongly non-polynomial growth rate (one could alternatively use this remark to present another equivalent definition of the strongly non-polynomial growth rate).

Finally, we also say that $a\in\mathcal{H}$ has trivial growth rate, if $\lim\limits_{t\to+\infty}a(t)=0$. Recall that we also included these functions when we defined the strongly non-polynomial functions.
We will now prove the following lemma.

\begin{lemma}\label{growthdecomposition}Let $\mathcal{H}$ be a Hardy field that contains the polynomials and
 let  $a_1,...,a_k\in\mathcal{H}$ be arbitrary functions. Then, the set $\mathcal{L}(a_1,...,a_k)$ of non-trivial linear combinations has a basis $(g_1,...,g_m,h_1,...,h_{\ell})$, where $m,\ell$ are non-negative integers, such that the functions $h_1,...,h_{\ell}$ have the form $p_i(t)+o_t(1)$, where $p_i$ is a real polynomial for every $ 1\leq i\leq \ell$ and $g_1,...,g_m$ have distinct and non-trivial strongly non-polynomial growth rates.   
 \end{lemma}
  
  \begin{proof}
     We can restrict our attention to the case that the functions $a_1,...,a_k$ are linearly independent (otherwise, we pass to a maximal subset of these functions whose elements are linearly independent). 
   We induct on $k$. For $k=1$, we have nothing to prove. Assume the claim holds for all integers smaller than $k$. All functions considered below are implicitly assumed to belong to $\mathcal{H}$.

   We may write each of the functions $a_1,..,a_k$ in the form $p_i(t)+x_i(t)$ where $p_i$ are real polynomials and $x_i(t)$ are strongly non-polynomial functions (either one of the functions $p_i,x_i$ may also be identically zero). After reordering, we may assume that \begin{equation*}
      x_1(t)\gg x_2(t)\gg \dots \gg x_{k}(t).
  \end{equation*}Now, we define the number $l\in \{0,1,...,k\}$ to be the smallest natural number, for which all functions $x_{l+1}(t),x_{l+2}(t)$ and so on have limit zero (as $t\to+\infty$). If none of the $x_i$ have limits going to $0$, then we just set $\ell=k$.
  
   We consider two cases.\\
i) If the functions $x_1,...,x_{l}$ have distinct growth rates, then we are done. In this case, the functions $g_j$ appearing in the statement are the functions $p_i(t)+x_i(t)$ for $1\leq i\leq l$, while the role of the functions $h_j$ is performed by the functions $p_i(t)+x_i(t)$ for $i>l$ (observe that for $i>l$, we have that $x_i(t)$ have trivial growth rate due to the definition of $l$). The strongly non-polynomial growth rates of the former set of functions are equal to the growth rates of the functions $x_1,..,x_{l}$, which are pairwise distinct. \\
ii) Assume now two of the functions among $x_1,...,x_l$ have the same growth rate. In particular, let $k_0$ be the smallest integer such that $x_{k_0}\sim x_{k_0+1}$ (obviously $k_0<l$) and let $r\geq 1$ be the largest integer such that 
$$x_{k_0}\sim x_{k_0+1}\sim  \cdots \sim x_{k_0+r}. $$ 
 For $k_0+1\leq i\leq k_0+r$, we can write $x_i(t)=x_{k_0}(t)+y_i(t)$, where $y_i(t)\prec x_i(t)$. Using this, we can write $a_{k_0}(t)=p_{k_0}(t)+x_{k_0}(t)$ and $$a_i(t)=(p_{k_0}(t)+x_{k_0}(t)) +(p_i(t)-p_{k_0}(t) +y_i(t)), \ \text{ for }  k_0+1\leq i\leq k_0+r.$$ Now we apply the induction hypothesis on the collection of functions \begin{multline*}\{ p_{k_0+1}(t)-p_{k_0}(t)+y_{k_0+1}(t),..., p_{k_0+r}(t)-p_{k_0}(t)+y_{k_0+r}(t),\\p_{k_0+r+1}(t)+x_{k_0+r+1}(t),...,p_{k}(t)+x_k(t)\}.\end{multline*} 
 
 This gives a basis $(g_1,...,g_{m},u_1,...,u_{\ell})$ for this set of functions, with the properties outlined in the statement. We add the functions $p_1(t)+x_1(t),..,p_{k_0}(t)+x_{k_0}(t)$ to the functions $g_1,...,g_m$ and add the functions\footnote{ Recall that $x_i(t)$ goes to 0 for $l< i\leq k$.} $p_i(t)+x_i(t)$, $l<i\leq k$,  to the collection $u_1,...,u_{\ell}$. In this way, we construct a basis for the original collection $a_1,...,a_k$ with the asserted properties (if the functions that we have constructed are not linearly independent, then we can just pass to a subset of these functions that will form a basis). Indeed, we only have to check that the functions $$p_1(t)+x_1(t),...,p_{k_0}(t)+x_{k_0}(t),g_1(t),...,g_m(t)$$ have distinct strongly non-polynomial growth rates. This follows by noting that the strongly non-polynomial growth rates of the functions $g_1,...,g_m$ cannot be larger than the growth rates of the functions $y_i$, which all grow strictly slower than $x_{k_0}$. Thus, the function $p_{k_0}(t)+x_{k_0}(t)$ has bigger strongly non-polynomial growth rate than all of the functions $g_1,...,g_m$. Furthermore, the strongly non-polynomial growth rate of the function $p_i(t)+x_i(t)\ (1\leq i\leq k_0)$ is the same as $x_i(t)$, and these are all pairwise distinct by the definition of $k_0$. The claim follows.
 \end{proof}

\begin{remark}i) Note that we do not require that the functions $a_1,...,a_k$ have polynomial growth in the above lemma.\\
ii) A very simple example that illustrates the above decomposition is the following: assume that we have the functions $a_1(t)=t^2+t^{3/2}, a_2(t)=t^{3/2}, a_3(t)=2t^{3/2}+t^2$ and $a_4(t)=t^{3/2}+t\log t +t^3$. These four functions are clearly linearly dependent. The above lemma provides the basis $(g_1,g_2,h_1)$, where $g_1(t)=t^{3/2}, g_2(t)=t\log t+t^3$ and $ h_1(t)=t^2$. The main property (which will be important in the proof of Theorem \ref{pointwise}) is that the functions $g_1,g_2$ have distinct strongly non-polynomial growth rates ($t^{3/2},t\log t$ respectively), even though $g_2$ grows like $t^3$ (i.e a polynomial). 
\end{remark}

\section{Nilmanifolds and quantitative equidistribution theory }\label{nilpotent crap}

\subsection{Background on nilmanifolds}\label{B1 subsection}

A large portion of the material concerning nilmanifolds (excluding the quantitative equidistribution results) can be found in \cite[Part 3]{Host-Kra structures}, where there is a focus on the ergodic theoretic point of view. For a more general presentation of the theory of nilpotent Lie groups, see also \cite{corwin}.

Let $G$ be a topological group.
A subgroup $H$ of a topological group $G$ is called discrete, if there is a cover of $H$ by open sets of $G$, such that each of these open sets contains exactly one element of $H$. It is called co-compact if the quotient topology makes $G/H$ a compact space. We call a subgroup with both of the above properties uniform and we will use the letters $\Gamma$ or $\Delta$ to denote such subgroups.

Let $G$ be a $k$-step nilpotent Lie group and $\G$ be a uniform subgroup.
The space $X=G/\G$ is called a $k$-step nilmanifold.

Let $b$ be any element in $G$. Then, $b$ acts on $G$ by left multiplication. Let $m_X$ be the image of the Haar measure of $G$ on $X$ under the natural projection map. Then, $m_X$ is invariant under the action of the element $b$ (and therefore the action of $G$). If we set $T(g\G)=(bg)\G$, then the transformation $T$ is called a nilrotation, and $(X,m_X, T)$ is called a nilsystem. If the transformation $T$ is ergodic, we say that $b$ acts ergodically on the nilmanifold $X$. It can be proven that $b$ acts ergodically on $X$ if and only the sequence $(b^nx)_{n\in\N}$ is dense on $X$ for all $x\in X$ (see, for instance, \cite[Chapter 11]{Host-Kra structures}).    

Let $x_n$ be a sequence of elements on $X=G/\G$. We say that $x_n$ is equidistributed on $X=G/\G$ if and only if for every continuous function $F:X\to \C$, we have \begin{equation*}
     \lim\limits_{N\to+\infty }\underset{1\leq n\leq N}{\E}F(x_n)=\int_X F d m_{X}
 \end{equation*}where $m_X$ is the (normalized) Haar measure of $X$. 

 A rational subgroup $H$ is a subgroup of $G $ such that $H\cdot e_X$ is a closed subset of $X=G/\G$, where $e_X$ is the identity element of $X$. Equivalently, $H\G$ is a closed subset of the space $G$. This, also, implies that $H$ must be closed in $G$ (see \cite[Chapter 10, Lemma 14]{Host-Kra structures}). A subnilmanifold  of $X$ is a set $Y\subset X$ of the form $H\cdot x$, where $x$ is an element of $X$ and $H$ is a rational subgroup of $G$.

\subsubsection{Horizontal torus and characters}
Assume $X=G/\G$ is a $k$-step nilmanifold with $G$ connected and simply connected and consider the subgroup $G_2=[G,G]$. The nilmanifold $Z=G/(G_2\G)$ is called the horizontal torus of $X$. We observe that $Z$ is a connected, compact Abelian Lie group, and thus isomorphic to some torus $\T^{d}$. For a $b\in G$, it can be shown that the nilrotation induced by $b$ is ergodic, if and only if the induced action of $b$ on $Z$ is ergodic \cite[Theorem 3]{parry} (see also the theorem in section 2.17 of \cite{Leibmanpointwise}).

A horizontal character $\chi$ is a continuous group morphism $\chi:G\to \C$, such that $\chi(g\gamma)=\chi(g)$ for all $\gamma\in\G$. We observe that $\chi$ also annihilates $G_2$ and therefore descends to the horizontal torus $Z$. Thus, under the natural projection map $\pi$, $\chi$ becomes a character on some torus $\T^d$. We will often use the notation $\chi\circ \pi$ when working in the horizontal torus, while we reserve the letter $\chi$ to denote the same character in the original group $G$.

\subsubsection{Change of base point}\label{change of base point}
For every $b\in G$, we have that the sequence $b^n\G$ is equidistributed in the set $ \{\overline{b^n\G\colon n\in \Z}\}$. Therefore, if $g$ is any other element in $G$, we have that the sequence $b^ng\G$ is equidistributed in the nilmanifold $g\overline{\{(g^{-1}bg )^n\G,n\in\N\}}$. This follows by noting that $b^ng=g(g^{-1}bg)^n$. An analogous relation holds for the elements of the set $(b^sg)_{s\in \R}$, which we define below. This trick, which is called the change of base point trick, can be used when we want to show that some sequence $v(n)x$ is equidistributed (on some specific nilmanifold depending on $x$) in order to change the base point $x$ to $\G$.

\subsubsection{Reduction to connected-simply connected Lie groups}
Let $G$ be a $k$-step nilpotent Lie group and let $\G$ be a uniform subgroup of $G$. Then, the space $X=G/\G$ is called a $k$-step nilmanifold. The space $X$ may have several representations of the form $G/\G$ (with possible variance in the degree of nilpotency). Let $G^{\circ}$ be the connected component of $e_G$ in $G$. If we assume that $G/G^{\circ}$ is finitely generated\footnote{Without loss of generality we can assume that in this article, because our results deal with the action of $G$ on finitely many elements of $X$.}, then, by passing to the universal cover $\tilde{G}$ of $G$, it can be shown that $X$ has a representation $\tilde{G}/\tilde{\G}$ where now the underlying group $\tilde{G}$ is simply connected. In addition, we can argue as in \cite[Section 1.11]{Leibmanpointwise} to deduce that $X$ can be embedded as a subnilmanifold in some nilmanifold $G'/\G'$, where $G'$ is a connected and simply connected nilpotent Lie group and every translation on $X$ has a representation in $X'=G'/\G'$. This means that for any $x\in X,\ b_1,\dots b_k\in G$ and continuous function $F:X\to \C$, we can find $x'\in X'$, $b_1',\dots,b'_k\in G'$ and $F':X'\to \C$, such that $F(b_1^{n_1}\dots b_k^{n_k}x)=   F'((b_1')^{n_1}\dots (b'_k)^{n_k}x)$ for all $n_1,\dots,n_k\in \Z$.

\subsection{Nilorbits and Ratner's theorem}\label{bs-Ratner subsection }

Let $G$ be a connected and simply-connected Lie group. It is well known that the exponential map $\exp $  from the Lie algebra of $G$ to $G$ is a diffeomorphism. In particular, it is a bijection between $G$ and its Lie algebra $\mathfrak{g}$. For $b\in G$ and $t\in \R$ we can then define the element $b^t$ as the unique element of $G$ satisfying $b^t=\exp(tX)$, where $\exp(X)=b$. As a corollary of Ratner's theorem \cite{Ratner}, we get the following:

\begin{lemma}\label{Ratnerlemma}
Let $G/\G$ be a nilmanifold with $G$ connected and simply connected. For any elements $b_1,...,b_k\in\G$, we have that the set \begin{equation*}
    \overline{b_1^{\R}\cdots b_k^{\R}\G}=\overline{\{b_1^{t_1}\cdots b_k^{t_k}\G\colon t_1,...,t_k\in \R\}}
\end{equation*}is a subnilmanifold of $X$ with a representation $H/\Delta$, for some closed, connected and rational subgroup $H$ of $G$ that contains the elements $b_1^s,...,b_k^s$ for all $s\in\R$ and $\Delta$ is a uniform subgroup of $H$.
\end{lemma}

We call the set $\{\overline{b^t\G\colon t\in \R}\}$ the nil-orbit of the element $b$. We will analogously denote by $\overline{b^{\Z}\G}$ the set $\{\overline{b^n\G\colon n\in \Z}\}$ and $\overline{b^{\N}\G}=\{\overline{b^n\G\colon n\in \N}\}$.

 We establish the following lemma, which will be necessary for our proofs.
 \begin{lemma}\label{ergodicaction}
Let $X=G/\G$ be a nilmanifold and let $b_1,...,b_k\in\G$ be any pairwise commuting elements. Then, there exists a real number $t$ such that \begin{equation*}
    \overline{b_1^{\R}\dots b_k^{\R}\G}= \overline{\{b_1^{n_1t}\dots b_k^{n_k t}\G\colon n_1,...,n_k\in \Z\}}.
\end{equation*}
\end{lemma} 
\begin{proof}
We want to find some $t\in \R$ so that the sequence $$\phi_t(n_1,...,n_k)=b_1^{n_1t}\dots b_k^{n_k t}$$ is equidistributed on the nilmanifold $Y=\overline{b_1^{\R}\dots b_k^{\R}\G}$. By Lemma \ref{Ratnerlemma}, $Y$ has a representation as $H/\Delta$, where $H$ is connected, simply connected and rational. Observe that $\phi_t$ naturally induces a $\Z^k$ action on $Y$ by $(\phi_t(n_1,...,n_k),h\Delta )\to b_1^{n_1t}\dots b_k^{n_k t}h\Delta$.
It is sufficient to show that this $\Z^k$-action is ergodic on $Y$, since this implies that $Y=\overline{\{\phi_t({\bf n})y,{\bf n}\in \Z^k\}}$ for all $y\in Y$. 
However, using the results in \cite{Leibmanpointwise} (specifically, Theorem 2.17), the above action is ergodic if and only if it is ergodic on the horizontal torus $Z$ of $Y$, which is homeomorphic to some torus $\T^d$. Equivalently, if we denote by $(b_{i,1},...,b_{i,d})$ the projection of the point $b_i\G$ on $Z$, then we need to check whether the sequence \begin{equation*}
    \big(t(n_1b_{i,1}+\dots +n_k b_{k,1}), ... ,t(n_1b_{1,d}+\dots+n_kb_{k,d})\big)
\end{equation*}is dense on $\T^d$. It suffices to choose $t$ so that $1/t$ is rationally independent of any integer combination of the coordinates $b_{i,j}$. This completes the proof.
\end{proof}

 \subsection{Polynomial sequences on nilmanifolds}
 
 We provide the general definition of polynomial sequences with respect to some filtration. 
 \begin{definition}\label{filtration}
 A filtration $G_{\bullet}$ of degree $d$ on a nilpotent Lie group $G$ is a sequence of closed connected subgroups $$G=G^{(0)}=G^{(1)}\supseteq G^{(2)}\supseteq \dots \supseteq G^{(d)}\supseteq G^{(d+1)}={e_G},$$such that $[G^{(i)},G^{(j)}]\subseteq G^{(i+j)}$ for all $i,j\geq 0$. The filtration is called rational if all groups $G^{(i)}$ appearing in the above sequence are rational subgroups of $G$. A polynomial sequence on $G$ with respect to the above filtration is a sequence $g(n)$ such that, for all positive integers $h_1,...,h_k$, we have that the sequence $\partial_{h_1}\dots \partial_{h_k} g$ takes values in $G^{(k)}$, for all $k\in \N$, where $\partial_h $ denotes the "differencing operator" that maps the sequence $(g(n))_{n\in \N}$ to the sequence $(g(n+h)(g(n))^{-1})_{n\in \N}$.
 \end{definition}
 
 An example of a filtration is the lower central series of the group $G$. For the purposes of this article, we will only need to consider polynomial sequences of the form \begin{equation}\label{polynomialsequencedef}
     v(n)=b_1^{p_1(n)}\cdot ...\cdot b_k^{p_k(n)}
\end{equation}where $b_i\in G$ for all $1\leq i\leq k$ and $p_i$ are real polynomials. Note that the terms $b_i^{p_i(n)}$ are well defined, due to our connectedness assumptions. To see that this is indeed a polynomial sequence with our initial definition, we construct a specific filtration on $G$. We assume that $G$ is $k$-step nilpotent and we also denote the maximum degree among the polynomials $p_i$ as $d$. We consider the filtration (of degree $dk$) $G_{\bullet}=(G^{(i)})_{0\leq i\leq dk}$, where $G^{(i)}=G_{\floor{i/d}+1}$ and $G_j$ are the commutator subgroups of $G$. This is a rational filtration because all commutator subgroups of $G$ are rational (see \cite[Chapter 10, Proposition 22]{Host-Kra structures} for the proof).
Then, the sequence $v(n)$ in \eqref{polynomialsequencedef} is a polynomial sequence with respect to this filtration. We direct the reader to the discussion after \cite[Corollary 6.8]{Green-Tao}, where these last observations were made originally.
We will also call the projected sequence $v(n)\G$ on $X=G/\G$ a polynomial sequence on $X$.

\subsection{Quantitative equidistribution}\label{quantitative subsection}

Assume that $p(t)$ is a polynomial. Then, $p(n)$ can be expressed uniquely in the form \begin{equation*}
    p(n)=\sum_{i=0}^{d}a_i n^i 
\end{equation*} for some real numbers $a_i$ and $d\in\N$. For $N\in \N$, we define the smoothness norm\footnote{The definition of the smoothness norms is a bit different in \cite{Green-Tao}. There, the authors write the polynomials in the form $p(n)=\sum_{i=0}^{d}a_i\binom{n}{i}$ and define the smoothness norm using the same definition as \eqref{smmothness norm} (the coefficients $a_i$ are different). However, these definitions give two equivalent norms and, thus, all theorems can be stated for both norms, up to changes in the absolute constants.}\begin{equation}\label{smmothness norm}
    \norm{e(p(n))}_{C^{\infty}[N]}=\max\limits_{1\leq i\leq d} (N^i\norm{a_i}_{\R/\Z} ).
\end{equation}

 A filtration on a Lie group $G$ gives rise to a basis on its Lie algebra $\mathcal{B}$, which is called a Mal'cev basis \cite{Malcev}. Mal'cev bases play an essential role in the theory of quantitative equidistribution on nilmanifolds. Firstly, we give the following definition:
 \begin{definition}\label{malcev}
 Let $X=G/\G$ be a $k$-step nilmanifold with a rational filtration $G_{\bullet}=(G^{(i)})_{i\geq 0}$. Define $m=\text{dim}(G)$ and $m_i=\text{dim}(G^{(i)})$. A basis $(\xi_1,...,\xi_m)$ of the associated Lie algebra $\mathfrak{g}$ over $\R$ is called a Mal'cev basis adapted to $G_{\bullet}$, if the following conditions are met:\\
 i) For each $0\leq j\leq m-1$, $\mathfrak{h}_j=\textit{span}(\xi_{h+1},...,\xi_m)$ is a Lie algebra ideal on $\mathfrak{g}$ and thus $H_j=\exp(\mathfrak{h}_j)$ is a normal Lie subgroup of $G$.\\
 ii) For every $0\leq i\leq k$, we have $G^{(i)}=H_{m-m_i}$.\\
 iii) Each $b\in G$ can be uniquely written in the form $\exp(t_1\xi_1)... \exp(t_m\xi_m)$ for $t_i\in \R$.\\
 iv) The subgroup $\G$ consists precisely of those elements which, when written in the above form, have all $t_i\in \Z$.
 \end{definition}

Suppose that the element $b$ is written in the form $\exp(t_1\xi_1)\cdots \exp(t_m\xi_m)$. The map $\psi:G\to \R^m $ defined by $\psi(b)= (t_1,...,t_m)$ is a diffeomorphism from $G$ to $\R^m$. The numbers $(t_1,...,t_m)$ are called the coordinates of $g$ with respect to the associated Mal'cev basis. If we consider the Euclidean metric on $\R^m$, we can construct a Riemannian metric $d_G$ on $G$, whose value at the origin is equal to the Euclidean metric of $\R^m$ at the origin (of $\R^{m}$) composed with the inverse map $\psi^{-1}$. This metric is invariant under right translations and induces a metric $d_X$ on $X=G/\G$ defined by the relation:\begin{equation*}
    d_X(g\G,h\G)=\inf\{ d_G(b,b'), bg^{-1}\in \G, b' h^{-1}\in \G  \}.
\end{equation*}
The metric used in \cite{Green-Tao} is slightly different than the one we consider here, but as the authors remark, these metrics are equivalent and all theorems hold as well by changing the absolute constants.

The sequence $(g(n)\Gamma)_{1\leq n\leq N}$ is said to be $\delta$-equidistributed on the nilmanifold $X=G/\Gamma$ if and only if for any Lipschitz function $F:X\to \C$, we have that \begin{equation*}
\big|    \underset{1\leq n\leq N}{\E} F(g(n)\Gamma)  -\int_X Fd\m_X                        \big|\leq \delta \norm{F}_{\text{Lip}(X)}
\end{equation*}where\begin{equation*}
    \norm{F}_{\text{Lip}(X)}=\norm{F}_{\infty} +\sup_{x,y\in X,\  x\neq y}^{}\frac{|F(x)-F(y)|}{d_X(x,y)}.
\end{equation*}

We now fix a $k$-step nilmanifold $X=G/\G$, as well as a positive integer $d$. We equip it with the rational filtration $G_{\bullet}$ of degree $dk$ that we defined above (after Definition \ref{filtration}), as well as a Mal'cev basis adapted to this filtration and the corresponding coordinate map $\psi: G\to \R^m$ ($m$ is the dimension of $G$). Observe that under this filtration, we have that $G^{(d+1)}=G_2$ and property ii) in Definition \ref{malcev} implies that $G_2=H_{m-m_{d+1}}$. Thus, the Mal'cev basis induces an isometric identification of the horizontal torus $Z=G/G_2\G$ with the torus $\T^{m-m_{d+1}}$ equipped with the standard metric.

Let $\pi: X\to Z$ denote the projection map and let $\chi$ be a horizontal character on $G$. Consider an element $b\in G$ with coordinates $(t_1,...,t_m)$. Then, by properties iii) and iv) in Definition \ref{malcev}, we have that there is some $ \overset{\rightarrow}{\ell}=(\ell_1,...,\ell_{m-m_{d+1}})\in \Z^{m-m_{d+1}}$ such that $$\chi\circ \pi (b)=\ell_1t_1+\dots +\ell_{m-m_{d+1}}t_{m-m_{d+1}}.$$Thus, we get a character on the torus $\T^{m-m_{d+1}}$(written here with additive notation). We can then define the modulus $\norm{\chi}$ of the character $\chi$ to be equal to \begin{equation}\label{moduli}
\norm{\overset{\rightarrow}{\ell}}=|\ell_1|+\dots +|\ell_{m-m_{d+1}}|.
\end{equation} 
If $v(n)$ is the polynomial sequence in \eqref{polynomialsequencedef} (recall that it is a polynomial sequence with respect to the filtration $G_{\bullet}$), then the sequence $\chi\circ \pi (v(n)\G)$ is a polynomial sequence on the horizontal torus $Z\cong \T^{m-m_{d+1}}$. Indeed, if we denote $\psi(b_i)=(t_{i,1},...,t_{i,m})$, then a simple calculation shows that\begin{multline*}
    \chi(\pi(v(n)\G))=\chi\big(\pi(b_1^{p_1(n)}\cdot ...\cdot b_k^{p_k(n)})\big)=\\p_1(n)(\ell_1 t_{1,1}+\dots+\ell_{m-m_{d+1}}t_{1,m-m_{d+1}})+\dots +p_k(n)(\ell_1 t_{k,1}+\dots+\ell_{m-m_{d+1}}t_{k,m-m_{d+1}}),
\end{multline*}which makes the fact that $\chi(\pi(v(n)\G))$ is a polynomial sequence more evident.

The primary tool that we shall use is the following theorem of Green-Tao which describes the orbits of polynomial sequences in finite intervals. We present it in the case of our filtration $G_{\bullet}$, although the statement holds for any rational filtration. Some quantitative information (specifically relating to the concepts of quantitative rationality of Mal'cev bases) has been suppressed, since in our applications the nilmanifold will be fixed and the above condition on the Mal'cev bases is guaranteed if we take $\delta$ small enough.

\begin{customthm}{F}   \cite[Theorem 2.9]{Green-Tao}    \label{quantitative}
Let $d$ be a non-negative integer, $X=G/\Gamma$ be a nilmanifold with $G$ connected and simply connected and we equip the nilmanifold $X$ with the Mal'cev basis adapted to the $dk$ filtration $G_{\bullet}$ as above. Assume $\delta$ is a sufficiently small (depending only on $X,d$) parameter. Then, there exist a positive constant $C=C(X,d)$ with the following property: For every $N\in\N$, if $(v(n))_{n\in \N} $ is a polynomial sequence with respect to $G_{\bullet}$ such that the finite sequence $(v(n)\Gamma)_{1\leq n\leq N}$ is not $\delta$-equidistributed, then for some non-trivial horizontal character $\chi$ (that depends on $N$ and the sequence $v(n)$) of modulus $\norm{\chi}\leq \delta^{-C}$ we have\begin{equation*}
    \norm{\chi(\pi (v(n)\G))}_{C^{\infty}(N)}\leq \delta^{-C},
\end{equation*}where $\pi$ denotes the projection map from $X$ to its horizontal torus.
\end{customthm}

In order to get a sense of how this theorem works, let us consider an application on a polynomial sequence on $\T$. Let $d$ be a positive integer and $\delta>0$ a small real number. Then, there exists a constant $C$ that depends only on $d$, such that for any polynomial $$p(t)=a_dx^d+\dots+a_1t+a_0$$of degree $d$, we have either that \begin{equation*}
    \Big|\E_{1\leq n\leq N}e(p(n))\Big|<\delta,
\end{equation*}or there exists an integer $q$ with $|q|\leq \delta^{-C}$, such that \begin{equation*}
    N^k\norm{qa_k}_{\T}\leq \delta^{-C}
\end{equation*}for every $1\leq k\leq d$.
Thus, either the exponential sums of the polynomial sequence $p(n)$ are small or the non-constant coefficients $a_k$ satisfy a "major-arc" condition (they are "close" to a rational with denominator bounded by $\delta^{-C}$). Observe that the constant $C$ does not depend on the length of the interval $N$.
\end{appendix}


\begin{thebibliography}{}







 
 
 
 
 
 

 
 
 
 
 
 
 
\bibitem{BerMorRic}

V. Bergelson, J. Moreira, F. Richter. Multiple ergodic averages along functions from a Hardy field: convergence, recurrence and combinatorial applications. arXiv:2006.03558 Preprint.









\bibitem{Boshernitzan1}

M. Boshernitzan. Uniform distribution and Hardy fields. {\em J. Anal. Math}. {\bf 62} (1994), 225--240




\bibitem{corwin}
L. Corwin and F. P. Greenleaf. {\em Representations of nilpotent Lie groups and their applications Part 1:
Basic theory and examples}. Cambridge Studies in Advanced Mathematics. {\bf 18} (1990)





















 
 
 \bibitem{Fraeq}
 N. Frantzikinakis. Equidistribution of sparse sequences on nilmanifolds. {\em J. Anal. Math}. {\bf 109} (2009), 353--395.
 
 
 
 
 
 
 
 
 
 













 


\bibitem{Fraopen}
N. Frantzikinakis. Some open problems on multiple ergodic averages.  {\em Bull. Hell. Math. Soc}. {\bf 60} (2016), 41--90. 








 \bibitem{Frajoint}
 N.Frantzikinakis. Joint ergodicity of sequences.
{\em Adv. in Math}. {\bf 417} (2023), (63pp). 


\bibitem{FraLeWie}
N. Frantzikinakis, E. Lesigne, M. Wierdl.
  Sets of $k$-recurrence but not $(k+1)$-recurrence.
    {\em Ann. de l'Inst. Fourier}. {\bf 56} (2006), no. 4, 839--849. 


 
























\bibitem{Fu1}
 H. Furstenberg. Ergodic behavior of diagonal measures and a theorem of Szemer\'{e}di on arithmetic progressions. {\em J. Anal. Math.} {\bf 71} (1977), 204--256.








\bibitem{Green-Tao}
B. Green, T. Tao. The quantitative behaviour of polynomial orbits on nilmanifolds. {\em Ann. of Math.}. {\bf 175} (2012). 465–-540.






\bibitem{Hardy 1}
G. H. Hardy. Properties of Logarithmico-Exponential Functions. {\em Proc. London Math. Soc. (2)} {\bf 10} (1912), 54--90
























\bibitem{Host-Kra1}
B. Host, B. Kra. Non-conventional ergodic averages and nilmanifolds. {\em Ann. of Math.} {\bf 161} (2005), no. 2, 397--488. 































\bibitem{Host-Kra structures}

B. Host, B. Kra. {\em Nilpotent Structures in Ergodic Theory}.  American Mathematical Society. {\bf 236} (2018), Mathematical Surveys and Monographs. 978-1-4704-4780-9.





























































\bibitem{Leibmanpointwise}

A. Leibman. Pointwise convergence of ergodic averages for polynomial sequences of rotations of a nilmanifold. {\em Ergodic Theory Dynam. Systems}. {\bf 25}. (2005), no. 1, 201–-213.
























\bibitem{Malcev}

A. Mal'cev. On a class of homogeneous spaces. {\em Izvestiya Akad. Nauk SSSR, Ser Mat.} {\bf 13} (1949), 9–-32.


\bibitem{parry}W. Parry. Dynamical systems on nilmanifolds. {\em  Bull. London Math. Soc.} {\bf 2} (1970),
37–-40




\bibitem{Ratner}

M. Ratner. Raghunatan’s topological conjecture and distribution of unipotent flows. {\em Duke Math. J}. {\bf 61}. (1991), no. 1, 235–-280.











\bibitem{Richter} F. K. Richter. Uniform distribution in nilmanifolds along functions from a Hardy field. {\em J. Anal. Math.} {\bf 149}. (2023), 421–-483



\bibitem{tsinas} K.Tsinas. Joint ergodicity of Hardy field sequences. {\em Trans. Amer. Math. Soc}. {\bf 376} (2023), 3191--3263 



























































\end{thebibliography}
\end{document}